\documentclass[letterpaper]{amsart}
\makeindex

\usepackage[alphabetic,initials,nobysame]{amsrefs}
\usepackage{amssymb,mathtools,mathrsfs,enumitem}
\usepackage{pgfplots}\pgfplotsset{compat=1.18}
\usepackage[bookmarksdepth=2]{hyperref}

\usepackage{pifont}
\newcommand{\cmark}{\textcolor{green!80!black}{\ding{51}}} % check mark
\newcommand{\xmark}{\textcolor{red}{\ding{55}}}   % x mark

\BibSpec{collection.article}{%
	+{}  {\PrintAuthors}				{author}
	+{,} { \textit}                  	{title}
	+{.} { }                         	{part}
	+{:} { \textit}                  	{subtitle}
	+{,} { \PrintContributions}      	{contribution}
	+{,} { \PrintConference}         	{conference}
	+{}  {\PrintBook}                	{book}
	+{,} { }                         	{booktitle}
	+{,} { }						 	{series}
	+{,} { }						 	{publisher}
	+{,} { }						 	{address}
	+{,} { \PrintDateB}              	{date}
	+{,} { pp.~}                     	{pages}
	+{,} { }                         	{status}
	+{,} { \PrintDOI}                	{doi}
	+{,} { available at \eprint}     	{eprint}
	+{}  { \parenthesize}            	{language}
	+{}  { \PrintTranslation}        	{translation}
	+{;} { \PrintReprint}            	{reprint}
	+{.} { }                         	{note}
	+{.} {}                          	{transition}
	+{}  {\SentenceSpace \PrintReviews} {review}
}

\theoremstyle{plain}
\newtheorem{theorem}{Theorem}
\newtheorem{proposition}[theorem]{Proposition}
\newtheorem{conjecture}[theorem]{Conjecture}
\newtheorem{problem}[theorem]{Problem}

\theoremstyle{definition}
\newtheorem{definition}[theorem]{Definition}

\numberwithin{equation}{section}
\numberwithin{theorem}{section}
\renewcommand{\bar}{\overline}
\newcommand{\R}{\mathbb R}
\newcommand{\C}{\mathbb C}
\newcommand{\D}{\mathbb D}

\newcommand{\N}{\mathbb N}

\DeclareMathOperator{\J}{\mathcal{J}}
\DeclareMathOperator{\dist}{{\mathrm{dist}}}
\DeclareMathOperator{\diam}{{\mathrm{diam}}}
\DeclareMathOperator{\id}{{\mathrm{id}}}

\DeclareMathOperator{\area}{\mathrm{Area}}
\DeclareMathOperator{\loc}{loc}

\newcommand{\cned}{\textit{CNED}}
\newcommand{\ned}{\textit{NED}}
\title{Uniformization problems in the plane: A survey}
\author{Dimitrios Ntalampekos}
\address{Department of Mathematics, Aristotle University of Thessaloniki, Thessaloniki, 54152, Greece.}
\thanks{The author was partially supported by the ERC Starting Grant, Grant Agreement no. 101214615, GRComPaS}
\email{dntalam@math.auth.gr}
\date{\today}

\begin{document}

	\begin{abstract}
		In this survey we present the history and recent progress on several fundamental (quasi)conformal uniformization problems in the complex plane. \textit{Uniformization} refers to the process of mapping a space to a canonical model by means of a well-behaved transformation that preserves the geometry and distorts shapes in a controlled fashion. A central problem in the area is Koebe's conjecture, which remains open after almost 120 years and predicts that each planar domain can be conformally mapped to a \textit{circle domain}---that is, a domain whose complementary components are points or closed disks. We trace the history of the conjecture, outline recent developments, and examine the associated uniqueness problem. We also discuss variants, with particular emphasis on the question whether a compact set can be mapped by a quasiconformal self-map of the plane to a \textit{Schottky set}---that is, a set in the plane whose complement is the union of disjoint open disks.
	\end{abstract}

\maketitle
\tableofcontents

\section{Introduction}
The object of this survey is to present the history and recent developments of uniformization problems in the complex plane. The term ``uniformization" originates from work of Koebe \cites{Koebe:uniformization1, Koebe:uniformization2, Koebe:uniformization3}, where the \textit{classical uniformization theorem}\index{classical uniformization theorem} is proved: every simply connected Riemann surface is conformally equivalent to the Riemann sphere, the complex plane, or the unit disk. The term also appears in another work of Koebe \cites{Koebe:Kreisnormierungsproblem, Koebe:FiniteUniformization} in relation with the \textit{Kreisnormierungsproblem}\index{Kreisnormierungsproblem} or else \textit{Koebe's conjecture}\index{Koebe conjecture}, which asserts that every domain in the Riemann sphere is conformally equivalent to a circle domain. We discuss this problem in detail in Section \ref{section:conformal_uniformization}. Therefore, according to Koebe, ``uniformization" refers to the conformal transformation of a surface or domain to a canonical object.

More generally, we may use other types of maps beyond the class of conformal maps. We are led to the following general problem for subsets of the plane or the sphere, which is the central object of the present work.
\begin{problem}\label{problem:general}
Can we transform a subset of the plane or the sphere into a \textbf{canonical set} with better geometric properties in a \textbf{controlled way}, without excessive distortion?
\end{problem}
Informally, the problem asks if a set can be transformed, as if it were made out of rubber, into a well-understood smoother set. For example, a \textit{canonical set} could be a disk, a circle, a line, an annulus, etc. 

One instance of such a result is the \textit{Riemann mapping theorem}, which asserts that every simply connected domain that is a proper subset of the complex plane can be transformed to the unit disk with a conformal map. Among its many important consequences, this classical theorem allows one to define harmonic measure on an arbitrary simply connected domain. Therefore, the existence of well-behaved transformations as in Problem \ref{problem:general} enables us to study the original set using tools that are available in the canonical, smoother set.

In Problem \ref{problem:general} we need to specify what it means to transform a set in a \textit{controlled way}. This essentially amounts to specifying the type of map used in the transformation. The most appropriate class of maps depends on the set under consideration. For instance, one could use bi-Lipschitz maps, which quasi-preserve lengths, or conformal maps, which preserve angles. However, these maps are not suitable for transforming a non-smooth fractal set, possibly of infinite length, into a smoother set. For this purpose, we need to switch to more general classes of maps. For example, we might use quasiconformal maps, which distort angles in a controlled manner, or quasisymmetric maps, which distort shapes in a controlled manner. More details about the latter two classes are presented in the following sections. 

We study Problem \ref{problem:general} in two general directions. In Section \ref{section:conformal_uniformization} we consider the problem of \textit{conformally} transforming a domain in the plane to a canonical domain, such as a disk, an annulus, or more generally a circle domain. This problem is known as Koebe's conjecture, as mentioned earlier. We present several partial results towards the conjecture in historical order. We also discuss the uniqueness of such conformal transformations, a problem that is related to the \textit{conformal rigidity} of circle domains and to the \textit{rigidity conjecture} of He and Schramm. Finally, we discuss potential approaches towards the general case of Koebe's conjecture.

In Section \ref{section:quasiconformal_uniformization} we present the problem of \textit{quasiconformally} transforming a set in the plane to a canonical set, such as a circle, a disk, a line, an annulus, or more generally a Schottky set, that is, a set whose complementary components are disks. We begin from the classical theorem of Ahlfors, which geometrically characterizes sets that are quasiconformally equivalent to circles, and we conclude with a recent result of the author that characterizes sets quasiconformally equivalent to Schottky sets. 

A fundamental difference exists between the two types of problems discussed in Sections \ref{section:conformal_uniformization} and \ref{section:quasiconformal_uniformization}. First, in the case of conformal uniformization (Section \ref{section:conformal_uniformization}), we note that a simply connected domain may possess a highly irregular or even fractal boundary. Nevertheless, the Riemann mapping theorem guarantees the existence of a conformal map from such a domain onto the unit disk, whose boundary is smooth. Therefore, in this setting the focus is on maps with nice behavior in the interior of a domain, but without imposing any control near the boundary.

On the other hand, in the case of quasiconformal uniformization (Section \ref{section:quasiconformal_uniformization}), we use exclusively quasiconformal maps defined in the entire plane or sphere. As a consequence, these maps distort the geometry in a controlled fashion globally. In particular, not every simply connected domain can be mapped to the unit disk with a quasiconformal self-map of the plane. A necessary and sufficient condition, expressed in terms of the geometry of the boundary, was established in a groundbreaking work of Ahlfors, which we present below.

Finally, we remark that uniformization by circle domains or Schottky sets is of interest to several areas beyond classical complex analysis, including the theory of hyperbolic surfaces, geometric group theory, complex dynamics, analysis on metric spaces, and probability theory. We discuss some of these connections in the next sections. 

\subsection*{Acknowledgments}
The author would like to thank Kai Rajala, Matthew Romney, and Malik Younsi for their comments and corrections, which improved the presentation.

\section{Conformal uniformization by circle domains}\label{section:conformal_uniformization}

\begin{figure}
	\centering
	\includegraphics[scale=1]{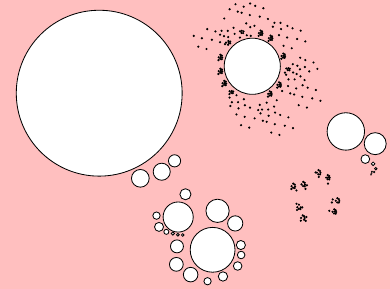}
	\caption{A circle domain. Its boundary can have isolated circles, circles converging to points, isolated points, points converging to circles, Cantor sets, etc.}\label{fig:circle_domain}
\end{figure}

\subsection{Koebe conjecture}\label{section:koebe}
The Riemann sphere\index{Riemann sphere} is denoted by $\widehat \C=\C\cup \{\infty\}$. A \textit{domain}\index{domain} in $\widehat \C$ is a connected open set. A domain $\Omega$ in the Riemann sphere $\widehat{\C}$ is called a \textit{circle domain}\index{circle domain}\index{domain!circle} if the connected components of $\partial \Omega$ are points and circles; see Figure \ref{fig:circle_domain}. The simplest instances of circle domains are the unit disk $\D=\{z\in \C: |z|<1\}$ and the entire plane $\C$, whose boundary consists of a single point at $\infty$.

We say that two domains $U,V\subset \widehat \C$ are \textit{conformally equivalent}\index{conformal equivalence} if there exists a conformal map $f$ from $U$ onto $V$.  Koebe posed the following deep problem, known as Koebe's conjecture or as the Kreisnormierungsproblem \cite{Koebe:Kreisnormierungsproblem}.

\begin{conjecture}[Koebe conjecture, 1908]\label{conjecture:koebe}\index{Koebe conjecture}\index{conjecture!Koebe}\index{Kreisnormierungsproblem}
Every domain in the Riemann sphere is conformally equivalent to a circle domain.
\end{conjecture}

The conjecture has an affirmative answer in the case of simply connected domains, a result known as the Riemann mapping theorem.

\begin{theorem}[Riemann mapping theorem, 1851]\index{Riemann mapping theorem}\label{theorem:riemann}
Every simply connected domain $\Omega\subsetneq \C$ is conformally equivalent to $\D$. Moreover, for any two conformal maps $f,g$ from $\Omega$ onto $\D$ the composition $f\circ g^{-1}$ is a M\"obius transformation.
\end{theorem}

Observe that M\"obius transformations\index{M\"obius transformation}\index{map!M\"obius} of $\widehat \C$, namely functions of the form 
$$f(z)=\frac{az+b}{cz+d},$$ map circles to circles and points to points. Thus, they map circle domains to circle domains. The next result generalizes the Riemann mapping theorem to the setting of finitely connected domains, a result due to Koebe \cite{Koebe:FiniteUniformization}. See Figure \ref{fig:koebe_finite}.

\begin{theorem}[Koebe uniformization theorem, 1920]\index{Koebe uniformization theorem}\label{theorem:koebe:uniformization}
Every finitely connected domain $\Omega\subset \widehat \C$ is conformally equivalent to a circle domain. Moreover, for any two conformal maps $f,g$ from $\Omega$ onto circle domains the composition $f\circ g^{-1}$ is a M\"obius transformation.
\end{theorem}

\begin{figure}
	\centering
	\begin{tikzpicture}
	\node at (0,0) {\includegraphics[scale=0.3]{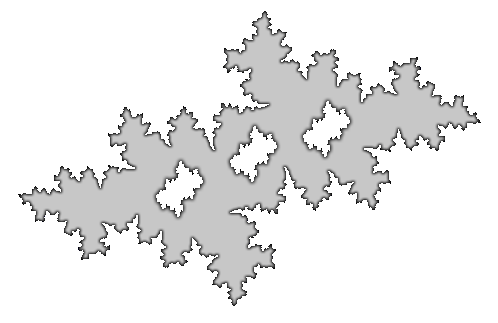}};
	\draw[->] (2.3,0)--(3.3,0);
	\draw[fill=black!20] (5,0) circle (1cm);
	\draw[fill=white] (5.4,0.4) circle (0.2cm);
	\draw[fill=white] (4.8,0.35) circle (0.3cm);
	\draw[fill=white] (5,-0.4) circle (0.35cm);
\end{tikzpicture}
	\caption{Illustration of Koebe's uniformization theorem for finitely connected domains. Note that a homeomorphism between domains preserves the number of boundary components.}\label{fig:koebe_finite}
\end{figure}
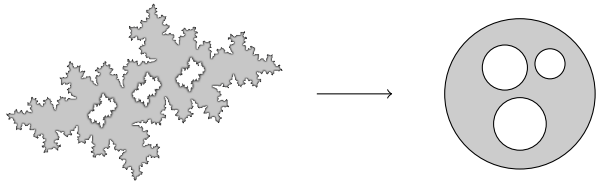

Proofs can be found in \cite{Conway:complex2}*{Theorem 15.7.9} and in \cite{Goluzin:complex}*{Section V.6}.

The conjecture was subsequently proved for certain symmetric domains by Koebe \cite{Koebe:symmetric} and for domains whose ``limit boundary components" satisfy certain conditions by Denneberg \cite{Denneberg:konforme}, Gr\"otzsch \cite{Grotzsch:koebe}, Sario \cite{Sario:koebe}, Meschowski \cites{Meschowski:koebe1, Meschowski:koebe2}, Strebel \cites{Strebel:koebe1, Strebel:koebe2}, Bers \cite{Bers:uniformization}, Sibner \cite{Sibner:KoebeQC}, Haas \cite{Haas:circle_domains} and others. 

The most significant result towards the conjecture so far was established in a seminal work of He and Schramm \cite{HeSchramm:Uniformization}. 

\begin{theorem}[He--Schramm uniformization theorem, 1993]\index{He--Schramm uniformization theorem}\label{theorem:he_schramm}
Every countably connected domain $\Omega\subset \widehat \C$ is conformally equivalent to a circle domain. Moreover, for any two conformal maps $f,g$ from $\Omega$ onto circle domains the composition $f\circ g^{-1}$ is a M\"obius transformation.
\end{theorem}

Alternative proofs of the theorem have been obtained by Schramm \cite{Schramm:transboundary} and Rajala \cite{Rajala:koebe}. Remarkably, in \cite{Schramm:transboundary} Schramm introduces the transboundary modulus\index{transboundary modulus}, a notion that has been proved to be of paramount importance for recent developments in uniformization and rigidity problems in the complex plane. 

We observe that in the above theorems, which establish special cases of Koebe's conjecture, we also have a uniqueness statement. Namely, up to postcomposition with M\"obius transformations, there exists a unique conformal map from $\Omega$ onto a circle domain. However, in the case of uncountably connected domains, uniqueness can fail. We discuss this phenomenon in Section \ref{section:rigidity}.

\subsection{Uniformization by other types of domains}

There are several results in the spirit of Koebe's conjecture that provide conformal transformation of a domain into special types of domains, beyond circle domains. We briefly discuss some of them.

A domain $\Omega\subset \widehat \C$ is a \textit{slit domain}\index{slit domain}\index{domain!slit} if it contains $\infty$ and the complementary components of $\Omega$ are points or rectilinear slits in a direction specified by an angle $\theta\in [0,\pi)$. The following classical result was proved by Hilbert \cite{Hilbert:slit} for finitely connected domains and generalized to infinitely connected domains by Gr\"otzsch \cite{Grotzsch:slit}. Proofs can be found in \cite{Courant:Dirichlet}*{Section II.3} and \cite{Goluzin:complex}*{Theorem V.2.1}.

\begin{theorem}\label{theorem:slit}
Every domain in the Riemann sphere is conformally equivalent to a slit domain. 
\end{theorem}

Therefore, the analogue of Koebe's conjecture is true in this case. We note that the conformal map onto a slit domain is not unique. However, the conformal map provided in the proof of Theorem \ref{theorem:slit} arises as the unique solution of a certain extremal problem. This fact motivates attempts to prove Koebe's conjecture by constructing and solving certain extremal problems, which automatically give rise to the desired conformal map. One such problem is considered by Schramm \cite{Schramm:transboundary_long} and provides an alternative proof of the conjecture in the countably connected case.

A very general uniformization result for finitely connected domains was proved by  Brandt--Harrington \cites{Brandt:conformal, Harrington:conformal}. This result allows one to prescribe all non-degenerate complementary components of the target domain up to homothety\index{homothety}\index{homothetic}, i.e., a map of the form $z\mapsto az+b$, where $a>0$ and $b\in \C$. The formulation given below is taken from \cite{Schramm:transboundary}*{Theorem 4.1}.

\begin{theorem}[\cites{Brandt:conformal, Harrington:conformal}]\label{theorem:brandt_harrington}
Let $\Omega\subset \widehat \C$ be a finitely connected domain. For each complementary component $b$ of $\Omega$, let $P_b\subset \C$ be a compact set that contains more than a single point such that $\widehat \C\setminus P_b$ is connected. Then there exists a conformal map $f$ from $\Omega$ onto a domain $D$ such that if $b$ is a non-degenerate complementary component of $\Omega$, then it corresponds under $f$ to a homothetic image of $P_b$.  
\end{theorem}

This result implies Koebe's uniformization theorem (Theorem \ref{theorem:koebe:uniformization}). Moreover, it implies that each finitely connected domain can be conformally mapped to a \textit{square domain}\index{square domain}\index{domain!square}, i.e., a domain whose non-degenerate complementary components are squares with sides parallel to the coordinate axes. Thanks to a result of Schramm \cite{Schramm:transboundary}*{Theorem 4.2}, Koebe's conjecture is equivalent to the statement that each domain is conformally equivalent to a square domain (or a domain bounded by equilateral triangles, etc.). 

Thus, one may attempt to approach Koebe's conjecture by studying conformal uniformization by square domains. Such domains appear more naturally than circle domains in extremal problems \cites{Schramm:Tiling, Schramm:transboundary_long, Bonk:square, Ntalampekos:CarpetsThesis}. Specifically, the recent work of Bonk \cite{Bonk:square} introduces a simple extremal problem for finitely connected domains whose solution gives a conformal map onto a square domain. This result is generalized by Solynin--Vidanage \cite{SolyninVidanage:rectangular} to domains bounded by rectangles. It would be interesting to extend these results to countably connected domains.

\subsection{Geometric conditions for uniformization}\label{section:geometric}

Koebe's conjecture has been established under some geometric conditions on a domain. We present here an incomplete collection of results in this spirit. 

\subsubsection{Uniform domains}\label{section:uniform}
A domain $\Omega\subset \C$ is called a \textit{uniform domain}\index{uniform domain}\index{domain!uniform} if there exists a constant $A\geq 1$ such that for any two points $z,w\in \Omega$ there exists curve $\gamma\colon [a,b]\to \Omega$ with $\gamma(a)=z$, $\gamma(b)=w$, $\ell(\gamma)\leq A|z-w|$, and 
$$\min\{\ell(\gamma|_{[a,t]}),\ell(\gamma|_{[t,b]})\}\leq A \dist(\gamma(t),\partial \Omega)\,\,\, \text{for every $t\in [a,b]$.}$$
Here we use the Euclidean metric and topology and $\ell(\gamma)$ denotes the length of $\gamma$. Uniform domains were introduced independently by Martio--Sarvas \cite{MartioSarvas:uniform} and Jones \cite{Jones:uniform} and appear nowadays in several problems of geometric function theory. They constitute a well-studied class that enjoys favorable geometric properties, similar to the strong properties of the unit disk \cite{Gehring:uniform_ubiquitous}.

The definition of a uniform domain implies that any two points can be connected by a twisted cone of aperture uniformly bounded from below. Geometrically, uniform domains are not allowed to have any inward or outward cusps. Martio and Sarvas \cite{MartioSarvas:uniform}*{Theorem 2.24} proved that each non-degenerate complementary component of a uniform domain is a quasidisk, an object with very strong geometric properties that resembles the unit disk from several viewpoints; we define and discuss quasidisks in detail in Section \ref{section:quasidisks}.

Koebe's conjecture was established for uniform domains by Herron and Koskela \cite{HerronKoskela:QEDcircledomains}, generalizing an earlier result of Herron \cite{Herron:uniform}. 
\begin{theorem}[\cites{HerronKoskela:QEDcircledomains, NtalampekosYounsi:rigidity}]\label{theorem:herron_koskela}
Every uniform domain $\Omega\subset \C$ is conformally equivalent to a circle domain. Moreover, for any two conformal maps $f,g$ from $\Omega$ onto circle domains the composition $f\circ g^{-1}$ is a M\"obius transformation.
\end{theorem}

The uniqueness part of the theorem was established much more recently by the author and Younsi \cite{NtalampekosYounsi:rigidity} and we discuss it in Section \ref{section:rigidity}.

\subsubsection{Cofat domains}\label{section:cofat}
As mentioned above, each non-degenerate complementary component of a uniform domain has nice geometry and in particular it is a quasidisk. Quasidisks satisfy the following strong geometric property. A set $A\subset \widehat \C$ is \textit{fat}\index{fat set} if there exists a constant $\tau>0$ such that for every $z\in A\cap \C$ and for every ball $B(z,r)$ that does not contain $A$ we have
$$\area(A\cap B(z,r))\geq \tau r^2.$$
In this case we say that $A$ is $\tau$-fat. Note that each point is trivially $\tau$-fat for all $\tau>0$. A domain $\Omega\subset \widehat \C$ is \textit{cofat}\index{cofat domain}\index{domain!cofat} if there exists $\tau>0$ such that each connected component of $\widehat \C\setminus \Omega$ is $\tau$-fat. Schramm proved that each quasidisk is fat \cite{Schramm:transboundary}*{Corollary 2.3}. In combination with the above mentioned result of Martio--Sarvas, we obtain that each uniform domain is cofat. Schramm generalized the result of Herron--Koskela and proved Koebe's conjecture for cofat domains \cite{Schramm:transboundary}.

\begin{theorem}[\cite{Schramm:transboundary}]\label{theorem:schramm_cofat}
Every cofat domain $\Omega\subset \widehat \C$ is conformally equivalent to a circle domain.
\end{theorem}

Unlike the previous results, we remark that uniqueness can fail in this case. Specifically, if $E=\widehat \C\setminus \Omega$ is a Cantor set then it is cofat trivially. If, in addition, $E$ has positive area, then we show in Proposition \ref{proposition:positive} below that uniqueness fails.

In the same work Schramm introduced the notion of transboundary modulus and used it, in combination with cofat domains, to provide an alternative proof of the He--Schramm uniformization theorem (Theorem \ref{theorem:he_schramm}). Since then, transboundary modulus has become a standard tool in modern geometric function theory, with numerous applications to uniformization and rigidity problems in the complex plane and metric spaces. The majority of the results presented in this survey rely on this notion. 

\subsubsection{Cospread domains}\label{section:cospread}
Very recently, Esmayli and Rajala \cite{EsmayliRajala:quasitripod} obtained an affirmative answer to Koebe's conjecture for domains whose complementary components a satisfy a different type of a uniform geometric condition. Before stating the theorem we give the necessary definitions. A homeomorphism $\phi\colon (X,d_X)\to (Y,d_Y)$ between metric spaces is called \textit{quasisymmetric}\index{quasisymmetric map}\index{map!quasisymmetric} if there exists a homeomorphism $\eta\colon [0,\infty)\to [0,\infty)$ such that for every triple of distinct points $x,y,z\in X$ we have
$$\frac{d_Y(\phi(x),\phi(y))}{d_Y(\phi(x),\phi(z))}\leq \eta\left(\frac{d_X(x,y)}{d_X(x,z)}\right).$$ 
In this case we say that $\phi$ is $\eta$-quasisymmetric. The homeomorphism $\eta$ is called a \textit{distortion function}\index{distortion function}. From a geometric point of view, quasisymmetric maps preserve relative sizes and shapes. In some instances, e.g., when $X=Y=\C$, quasisymmetric maps coincide with quasiconformal maps, defined in Section \ref{section:qc}. In general, quasisymmetric maps can be regarded as a generalization of quasiconformal maps in metric spaces. 

\begin{figure}
	\includegraphics[scale=0.25]{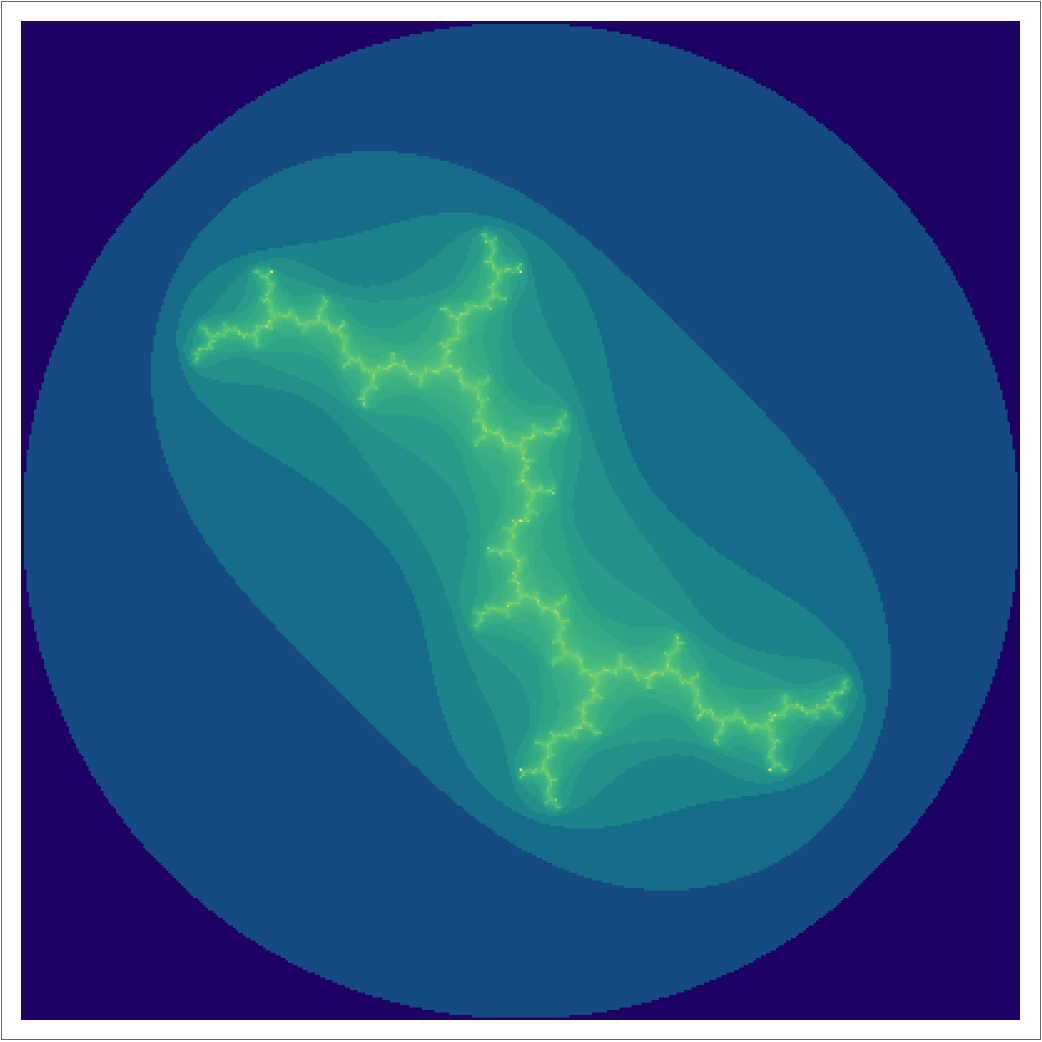}
	\caption{The Julia set of the polynomial $z^2+i$ is an example of an $\eta$-spread set because it contains $\eta$-quasitripods in all locations and scales.}\label{fig:julia}
\end{figure}

The \textit{standard tripod}\index{tripod} $T_0\subset \C$ is the union of the three segments from $0$ to $e^{2\pi i k/3}$, $k=0,1,2$. A set $T\subset \C$ is an $\eta$-\textit{quasitripod} if there exists an $\eta$-quasisymmetric homeomorphism $\phi\colon T_0\to T$. A set $A\subset \widehat \C$ is called $\eta$-spread if for every $z\in A\cap \C$ and every ball $B(z,r)$ that does not contain $A$ there is an $\eta$-quasitripod $T\subset A\cap B(z,r)$ with $\diam(T)\geq r/\eta(1)$. A domain $\Omega\subset \widehat \C$ is \textit{cospread}\index{cospread domain}\index{domain!cospread} if there exists a distortion function $\eta$ such that every component of $\widehat{\C}\setminus \Omega$ is $\eta$-spread.  We now state \cite{EsmayliRajala:quasitripod}*{Corollary 1.6}.

\begin{theorem}[\cite{EsmayliRajala:quasitripod}]\label{theorem:esmayli_rajala}
Every cospread domain $\Omega\subset \widehat \C$ is conformally equivalent to a circle domain.
\end{theorem}

We remark that uniqueness might fail here as in the case of Schramm's theorem above. The proof of Theorem \ref{theorem:esmayli_rajala} relies on Schramm's transboundary modulus, which is used in the proof of Theorem \ref{theorem:schramm_cofat}. However, unlike cofat domains, whose non-degenerate complementary components have positive area, in cospread domains the complementary components can be very thin and may have vanishing area; see Figure \ref{fig:julia}. This makes the proof significantly more involved than that of Schramm's theorem.

\subsubsection{Gromov hyperbolic domains}
Observe that the geometric conditions of Theorems \ref{theorem:herron_koskela}, \ref{theorem:schramm_cofat}, and \ref{theorem:esmayli_rajala} are not conformally invariant. That is, if $\Omega$ satisfies one of the conditions in these theorems, then conformal images of $\Omega$ need not satisfy the same condition. For instance, if $\Omega=\D$, then all we can say about conformal images of $\Omega$ is that they are simply connected, but they need not satisfy any strong Euclidean geometric condition. Recently, Karafyllia and the author \cite{KarafylliaNtalampekos:gromov_hyperbolic} proved Koebe's conjecture for the class of Gromov hyperbolic domains, which is a conformally invariant class that contains all finitely connected domains. We now give the formal definition. 

A geodesic metric space $X$ is \textit{Gromov hyperbolic}\index{Gromov hyperbolic space} if there exists $\delta>0$ such that for every geodesic triangle in $X$ each side is within the $\delta$-neighborhood of the union of the other two sides; see Figure \ref{fig:gromov}. Despite the simplicity of this definition, Gromov \cite{Gromov:hyperbolic} proved that several features of hyperbolic space can be recovered by this condition. Gromov hyperbolic spaces are an object of study in modern geometric group theory in connection with the uniformization problem; see the survey articles \cites{Bonk:icm, Kleiner:icm,  Ntalampekos:surface_survey}.

\begin{figure}
	\centering
	\includegraphics[scale=1]{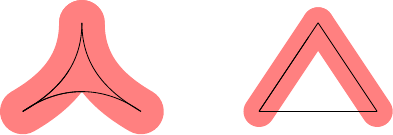}
	\caption{Left: A geodesic triangle in a Gromov hyperbolic space with the $\delta$-neighborhood of two sides containing the third side. Right: A geodesic triangle in the Euclidean plane, which is not Gromov hyperbolic.}\label{fig:gromov}
\end{figure}

Bonk, Heinonen and Koskela \cite{BonkHeinonenKoskela:gromov_hyperbolic} studied domains in the Riemann sphere that are Gromov hyperbolic when equipped with the quasihyperbolic metric. Let $\Omega\subsetneq \widehat \C$ be a domain, equipped with the spherical metric. The \textit{quasihyperbolic metric}\index{quasihyperbolic metric}\index{metric!quasihyperbolic} on $\Omega$ is defined as
$$k_\Omega(z,w)=\inf_\gamma \int_\gamma \frac{1}{\dist_\sigma(z,\partial \Omega)} \frac{2|dz|}{1+|z|^2},$$
where the infimum is taken over all rectifiable curves $\gamma$ with endpoints $z,w$; here $\sigma$ denotes the spherical metric. We say that $\Omega$ is Gromov hyperbolic\index{Gromov hyperbolic domain}\index{domain!Gromov hyperbolic} if the metric space $(\Omega,k_\Omega)$ is Gromov hyperbolic. By a result of Buckley and Herron \cite{BuckleyHerron:geodesics}*{Theorem B}, for hyperbolic domains in $\C$ (i.e., with at least two boundary points in $\C$) an equivalent condition is that the space $(\Omega,h_\Omega)$ is Gromov hyperbolic, where $h_\Omega$ is the hyperbolic metric on $\Omega$. From this result we obtain immediately the conformal invariance of Gromov hyperbolicity; see also \cite{KarafylliaNtalampekos:gromov_hyperbolic}*{Theorem 4.1} for an alternative approach. 

Examples of Gromov hyperbolic domains include all simply connected domains, all finitely connected domains, and all uniform domains. Intuitively, Gromov hyperbolic domains need not have good Euclidean geometry (for example, think of an arbitrary simply connected domain), but they have good hyperbolic geometry. For example, it is shown in \cite{BonkHeinonenKoskela:gromov_hyperbolic}*{Section 7} that each Gromov hyperbolic domain $\Omega$ has the \textit{Gehring--Hayman property}\index{Gehring--Hayman property}: there exists a constant $C\geq 1$ such that if $\gamma$ is a geodesic in the space $(\Omega,k_\Omega)$, then any curve $\beta$ in $\Omega$ with the same endpoints as $\gamma$ has spherical length
$$\ell_\sigma(\beta)\geq C^{-1}\ell_\sigma(\gamma).$$
Thus, quasihyperbolic geodesics have almost minimal length. A characterization of Gromov hyperbolic domains that involves the Gehring--Hayman property and a separation property is established by Balogh and Buckley in \cite{BaloghBuckley:gromov}.

We now state the main theorem of recent work of Karafyllia and the author \cite{KarafylliaNtalampekos:gromov_hyperbolic}, which establishes Koebe's conjecture for Gromov hyperbolic domains.

\begin{theorem}[\cite{KarafylliaNtalampekos:gromov_hyperbolic}]\label{theorem:karafyllia_ntalampekos}
Every Gromov hyperbolic domain $\Omega\subset \widehat \C$ is conformally equivalent to a uniform circle domain. Moreover, for any two conformal maps $f,g$ from $\Omega$ onto circle domains the composition $f\circ g^{-1}$ is a M\"obius transformation.
\end{theorem}

\begin{figure}
	\centering
	\begin{tikzpicture}
	\node at (0,0) {\includegraphics[scale=0.3]{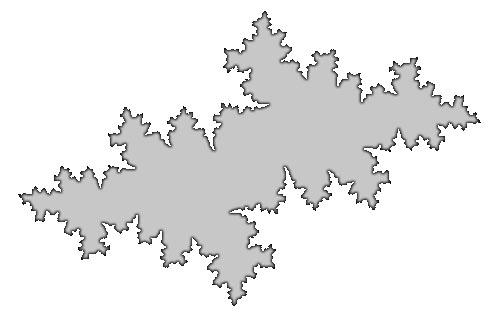}};
	\draw[<->] (2.3,0)--(3.3,0);
	\draw[fill=black!20] (5,0) circle (1cm);
\end{tikzpicture}
	\caption{Simply connected domains have well-behaved hyperbolic geometry, thanks to the Riemann mapping theorem, but their Euclidean geometry may be highly irregular. On the other hand, the unit disk enjoys both well-behaved Euclidean and hyperbolic geometry. This analogy extends to Gromov hyperbolic and uniform domains by Theorem \ref{theorem:gromov_uniform}.}\label{fig:simply_connected}
\end{figure}

The uniqueness statement follows from \cite{NtalampekosYounsi:rigidity} as in the case of uniform domains in Theorem \ref{theorem:herron_koskela}. This theorem gives a further insight. Although Gromov hyperbolic domains do not necessarily have good Euclidean geometry, they can be conformally transformed to uniform domains, which satisfy very strong geometric properties, as discussed in Section \ref{section:uniform}; see Figure \ref{fig:simply_connected}. Conversely, it was earlier shown by Bonk--Heinonen--Koskela \cite{BonkHeinonenKoskela:gromov_hyperbolic}*{Theorem 1.11} that uniform domains and conformal images of such domains are Gromov hyperbolic. Thus, we obtain the following consequence.

\begin{theorem}[\cites{BonkHeinonenKoskela:gromov_hyperbolic, KarafylliaNtalampekos:gromov_hyperbolic}]\label{theorem:gromov_uniform}
Gromov hyperbolic domains in $\widehat \C$ are precisely the conformal images of uniform domains.
\end{theorem}

We note that analogues of Koebe's conjecture have been studied on domains contained in metric surfaces\index{metric surface}, that is, topological surfaces equipped with a metric \cites{MerenkovWildrick:uniformization, RajalaRasimus:qs_koebe, Rehmert:thesis, HakobyanLi:qs_embeddings, LiRajala:cofat}. These results go beyond the scope of the present survey.

In Table \ref{tab:koebe} we summarize the results presented in this section.

\begin{table}
    \centering
    {\footnotesize
    \begin{tabular}{|c|c|c|}
        \hline
        \textbf{Types of domains} & \textbf{Existence} & \textbf{Uniqueness} \\
        \hline
        Simply connected & \cmark Riemann & \cmark  \\
        \hline
        Finitely connected & \cmark Koebe \cite{Koebe:FiniteUniformization} & \cmark Koebe \cite{Koebe:FiniteUniformization} \\
        \hline
        Countably connected & \cmark He--Schramm \cite{HeSchramm:Uniformization} & \cmark He--Schramm \cite{HeSchramm:Uniformization} \\
        \hline
        Uniform & \cmark  Herron--Koskela \cite{Herron:uniform} & \cmark Ntalampekos--Younsi \cite{NtalampekosYounsi:rigidity} \\
        \hline
        Cofat & \cmark Schramm \cite{Schramm:transboundary} & \xmark Proposition \ref{proposition:positive}\\
        \hline
        Cospread & \cmark Esmayli--Rajala \cite{EsmayliRajala:quasitripod} & \xmark Proposition \ref{proposition:positive}\\
        \hline
        Gromov hyperbolic & \cmark Karafyllia--Ntalampekos \cite{KarafylliaNtalampekos:gromov_hyperbolic} & \cmark Ntalampekos--Younsi \cite{NtalampekosYounsi:rigidity}\\
        \hline
    \end{tabular}
    }
    \medskip
    \caption{Existence and uniqueness results on Koebe's conjecture.}
    \label{tab:koebe}
\end{table}

\subsection{Conformal rigidity}\label{section:rigidity}

A circle domain $D$ is \textit{conformally rigid}\index{conformally rigid domain}\index{domain!conformally rigid} if every conformal map from $\Omega$ onto another circle domain is the restriction of a M\"obius transformation. Note that if there exists a conformal map $f$ from a domain $\Omega$ onto a circle domain $D$, then $f$ is unique up to postcomposition with M\"obius transformations if and only if the circle domain $D$ is conformally rigid; see Figure \ref{fig:rigidity}. Therefore, conformally rigid circle domains are closely related to the uniqueness of conformal maps as in Koebe's conjecture. 

Based on Theorem \ref{theorem:he_schramm}, every countably connected circle domain is conformally rigid. However, not every circle domain is conformally rigid. In the next statement we see that circle domains whose boundary is a Cantor set of positive area are not rigid.

\begin{figure}
	\includegraphics[scale=1]{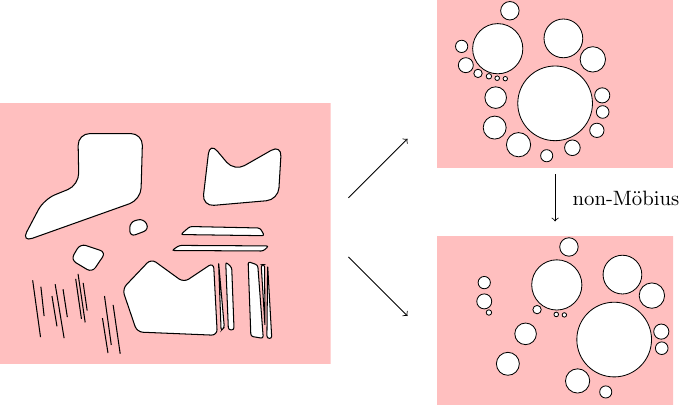}
	\caption{Two conformal maps from a domain to circle domains. These conformal maps do not differ by a M\"obius transformation if and only if the circle domains are not conformally rigid.}\label{fig:rigidity}
\end{figure}

\begin{proposition}\label{proposition:positive}
Let $E\subset \C$ be a totally disconnected compact set with positive area. There exists a homeomorphism $f\colon\C\to \C$ that is conformal in $\C\setminus E$ but it is not a M\"obius transformation.
\end{proposition}
\begin{proof}
According to a result of Uy \cite{Uy:lipschitz}, if $E$ has positive area, then there exists a non-constant bounded analytic function $g\colon \C\setminus E\to \C$ that is Lipschitz continuous; in particular, there exists $L>0$ such that $|g(z)-g(w)|\leq L|z-w|$ for all $z,w\in \C\setminus E$. Since $E$ is totally disconnected, $g$ extends continuously to a function in $\C$ that satisfies $|g(z)-g(w)|\leq L|z-w|$. We now define  
$$f(z)=z+ \frac{1}{2L}g(z), \quad z\in \C.$$
Obviously, $f$ is analytic in $\C\setminus E$. If $f$ were a M\"obius transformation, then $g$ would be analytic in $\C$ and hence constant by Liouville's theorem, a contradiction. Thus, $f$ is not a M\"obius transformation. Note that 
$$|f(z)-f(w)|\geq |z-w| -\frac{1}{2L}|g(z)-g(w)| \geq \frac{1}{2}|z-w|$$
for $z,w\in \C$, so $f$ is injective and in fact it is a homeomorphism of $\C$. 
\end{proof}

More generally, using the technique of quasiconformal deformation of Schottky groups, introduced by Sibner \cite{Sibner:KoebeQC}, it can be shown that any circle domain whose boundary has positive area is not conformally rigid. See also \cite{Younsi:RemovabilityRigidityKoebe}*{Lemma 18} for another proof based on Sibner's technique.

The conformal rigidity problem was studied thoroughly by He and Schramm \cite{HeSchramm:Rigidity}, who gave a sufficient condition for rigidity involving the size of the boundary of  a circle domain.

\begin{theorem}[\cite{HeSchramm:Rigidity}]\label{theorem:he_schramm_rigid}
If $\Omega\subset \widehat{\C}$ is a circle domain whose boundary has $\sigma$-finite length, then $\Omega$ is conformally rigid. 
\end{theorem}

Here, the length of a set is the Hausdorff $1$-measure; see \cite{Folland:real}*{Section 11.2} or \cite{Heinonen:metric}*{Section 8.3} for the definition. Obviously, there are at most countably many circles in the boundary of a circle domain, so their total length is $\sigma$-finite. However, the point components in the boundary of a circle domain can form a very large set; for example, a Cantor set of positive area could be part of the boundary. Thus, the obstacle for the rigidity of a circle domain seems to be the size of this set. He and Schramm posed the following conjecture.

\begin{conjecture}[Rigidity conjecture, 1994]\index{rigidity conjecture}\index{conjecture!rigidity}\label{conjecture:he_schramm}
A circle domain $\Omega$ is conformally rigid if and only if its boundary $\partial \Omega$ is conformally removable. 
\end{conjecture}

Here a compact set $E\subset \widehat \C$ is \textit{conformally removable}\index{conformally removable set}\index{removable} if every homeomorphism of $\widehat \C$ that is conformal in $\widehat \C\setminus E$ is a M\"obius transformation. Conformally removable sets include sets of $\sigma$-finite length due to Besicovitch \cite{Besicovitch:Removable} and boundaries of uniform domains due to Jones \cite{Jones:removability}. See also \cites{JonesSmirnov:removability, KoskelaNieminen:quasihyperbolic, Younsi:RemovabilityRigidityKoebe, Ntalampekos:metric_definition, Ntalampekos:metric_removable} for other sufficient geometric conditions. Non-removable sets include the Sierpi\'nski gasket and Sierpi\'nski carpets \cites{Ntalampekos:gasket, Ntalampekos:CarpetsNonremovable}. We direct the reader to the surveys \cites{Younsi:removablesurvey, Bishop:hard} for further details and open problems on removability. 

Some positive evidence towards Conjecture \ref{conjecture:he_schramm} was provided by Younsi \cite{Younsi:RemovabilityRigidityKoebe}. Further evidence was provided by the author and Younsi \cite{NtalampekosYounsi:rigidity}.

\begin{theorem}[\cite{NtalampekosYounsi:rigidity}]\label{theorem:ntalampekos_younsi}
Let $\Omega\subset \widehat \C$ be a circle domain and suppose that for a point $z_0\in \Omega$ we have $\int_\Omega k_\Omega(z_0,z)^2\, d\Sigma(z)<\infty$. Then $\Omega$ is conformally rigid. 
\end{theorem}

Here $\Sigma$ denotes the spherical Lebesgue measure. The condition of the theorem had been shown by Jones and Smirnov \cite{JonesSmirnov:removability} to imply the conformal removability of $\partial \Omega$. Moreover, this condition holds for uniform circle domains, thus Theorem \ref{theorem:ntalampekos_younsi} implies the uniqueness statements in Theorem \ref{theorem:herron_koskela} and Theorem \ref{theorem:karafyllia_ntalampekos}.

In \cite{Ntalampekos:metric_definition} the author introduced the notion of sets that are \textit{countably negligible for extremal distance}\index{countably negligible for extremal distance}\index{CNED}, abbreviated as $\cned$ sets, and proved that such sets are conformally removable. This notion resembles and generalizes the classical notion of sets that are \textit{negligible for extremal distance}, abbreviated as $\ned$ sets, which were studied by Ahlfors and Beurling \cite{AhlforsBeurling:Nullsets}. In order to avoid technicalities, we do not define these notions.  In \cite{Ntalampekos:cned} the author showed that the class of $\cned$ sets includes sets of $\sigma$-finite length and boundaries of domains satisfying the condition of Theorem \ref{theorem:ntalampekos_younsi}. Finally, in \cite{Ntalampekos:rigidity_cned}, he showed the rigidity of circle domains whose boundary is a $\cned$ set, a result that we state below.

\begin{theorem}[\cite{Ntalampekos:rigidity_cned}]\label{theorem:ntalampekos_cned}
Let $\Omega\subset \widehat \C$ be a circle domain whose boundary is a $\cned$ set. Then $\Omega$ is conformally rigid. 
\end{theorem}

This result implies Theorem \ref{theorem:he_schramm_rigid}, Theorem \ref{theorem:ntalampekos_younsi}, and provides strong evidence for Conjecture \ref{conjecture:he_schramm}.
Nevertheless, Conjecture \ref{conjecture:he_schramm} was subsequently disproved by Rajala \cite{Rajala:rigidity}.

\begin{theorem}[\cite{Rajala:rigidity}]\label{theorem:rajala}
There exists a conformally rigid circle domain whose boundary is not conformally removable. 
\end{theorem}

We note that if $\Omega$ is a conformally rigid circle domain with totally disconnected boundary, then its boundary is removable, as an immediate consequence of the definitions. So, any example that disproves this direction of the conjecture, as the one given by Rajala, must have boundary that contains both circles and points. 

The rigidity of the circle domain constructed by Rajala involves a metric characterization of conformal maps established by the author \cite{Ntalampekos:metric_definition}. The fact that the boundary is not removable follows from a theorem of Wu \cite{Wu:cantor}, which states that the product of a Cantor set $E\subset \R$ with a sufficiently thick Cantor set $F\subset \R$ is not removable. 

We remark that the reverse direction of Conjecture \ref{conjecture:he_schramm} remains open. The author has conjectured in \cite{Ntalampekos:cned} that conformal removability is equivalent to the $\cned$ condition. If this conjecture is true, then in Conjecture \ref{conjecture:he_schramm} the removability of $\partial \Omega$ implies the rigidity of $\Omega$. 

In Table \ref{tab:rigidity} we summarize several results related to the conformal rigidity of a circle domain $\Omega$ and the removability of $\partial \Omega$.

\begin{table}
    \centering
    {\footnotesize
    \begin{tabular}{|c|c|c|}
        \hline
        \textbf{Types of circle domains} & $\Omega$ \textbf{rigid} & $\partial \Omega$ \textbf{removable} \\
        \hline
        Simply connected & \cmark   & \cmark Morera \\
        \hline
        Finitely connected & \cmark Koebe \cite{Koebe:FiniteUniformization} & \cmark Morera \\
        \hline
        Countably connected & \cmark He--Schramm \cite{HeSchramm:Uniformization} & \cmark Besicovitch \cite{Besicovitch:Removable}  \\
        \hline
        Boundary has $\sigma$-finite length & \cmark He--Schramm \cite{HeSchramm:Rigidity}  & \cmark Besicovitch \cite{Besicovitch:Removable}\\
        \hline
        Uniform & \cmark Ntalampekos--Younsi \cite{NtalampekosYounsi:rigidity} & \cmark Jones \cite{Jones:removability} \\
        \hline
        Quasihyperbolic distance in $L^2$ &\cmark Ntalampekos--Younsi \cite{NtalampekosYounsi:rigidity} & \cmark Jones--Smirnov \cite{JonesSmirnov:removability} \\
        \hline
        Boundary is $\cned$ & \cmark Ntalampekos \cite{Ntalampekos:rigidity_cned} & \cmark Ntalampekos \cite{Ntalampekos:metric_definition}\\
        \hline
        Boundary has positive area & \xmark Sibner \cite{Sibner:KoebeQC} & \xmark Proposition \ref{proposition:positive}\\
        \hline
        Rajala's example & \cmark Rajala \cite{Rajala:rigidity} & \xmark Wu \cite{Wu:cantor}\\
        \hline
    \end{tabular}
    }
    \medskip
    \caption{Results related to the conformal rigidity of a circle domains and the removability of its boundary.}
    \label{tab:rigidity}
\end{table}

\subsection{Uniformization by exhaustion}
We discuss some approaches towards the general case in Koebe's conjecture. 

The central idea, which has been utilized for proving several cases of the conjecture, is to approximate a given domain $\Omega$ by a sequence of finitely connected circle domains $\{\Omega_n\}_{n\in \N}$, and then use Koebe's uniformization theorem (Theorem \ref{theorem:koebe:uniformization}) to find a sequence of conformal maps $f_n$ from $\Omega_n$ onto a finitely connected circle domain $D_n$, $n\in \N$. 

The circle domains $\Omega_n$, $n\in \N$, are taken so that they converge to the original domain $\Omega$ in the sense of \textit{kernel convergence of Carath\'eodory}\index{kernel!convergence}\index{Carath\'eodory kernel convergence} with respect to a fixed base point. We recall this notion of convergence. Let $z_0\in \Omega$ and let $\{\Omega_n\}_{n\in \N}$ be a sequence of domains such that $z_0\in \Omega_n$ for all $n\in \N$. We say that $\Omega$ is the $z_0$-\textit{kernel}\index{kernel} of $\{\Omega_n\}_{n\in \N}$ if $\Omega$ is the largest domain with the property that $z_0\in \Omega$ and each compact set $K\subset \Omega$ is contained in $\Omega_n$ for all sufficiently large $n\in\N$. We say that $\{\Omega_n\}_{n\in \N}$ \textit{converges to $\Omega$ in the Carath\'eodory sense with base at} $z_0$ if $\Omega$ is the $z_0$-kernel of every subsequence of $\{\Omega_n\}_{n\in \N}$.

Now, suppose that $\{\Omega_n\}_{n\in \N}$ is as above and for each $n\in \N$ we have a conformal map $f_n\colon \Omega_n\to D_n$ onto a finitely connected circle domain $D_n$. If the sequence $\{f_n\}_{n\in \N}$ is appropriately normalized, then by normality criteria we may pass to a subsequence that converges locally uniformly to a conformal map $f$ on $\Omega$. By a version of Carath\'eodory's kernel convergence theorem (see \cite{Goluzin:complex}*{Theorem V.5.1}) for multiply connected domains, we have $f(\Omega)=D$, where $D$ is the Carath\'eodory limit of $\{D_n\}_{n\in \N}$ with base at $f(z_0)$. The last step towards proving Koebe's conjecture is to verify that the domain $D$ is a circle domain. 

This is actually a highly nontrivial task and a naive approach does not work because a sequence of circle domains need not converge to a circle domain in the Carath\'eodory sense! See Figure \ref{fig:caratheodory} for an illustration of this phenomenon.

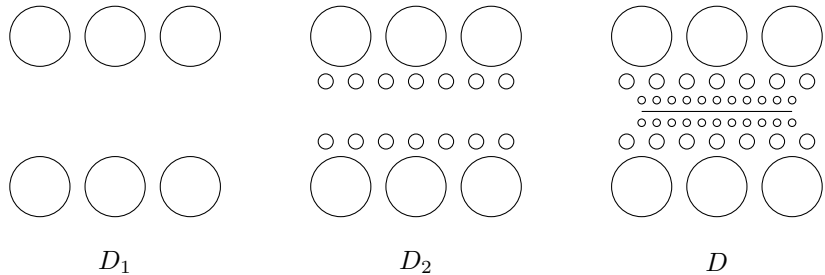
\begin{figure}
	\centering
	\begin{tikzpicture}
	\begin{scope}
		\draw (-1,1) circle (0.4cm);\draw (-1,-1) circle (0.4cm);
		\draw (0,1) circle (0.4cm);\draw (0,-1) circle (0.4cm);
		\draw (1,1) circle (0.4cm);\draw (1,-1) circle (0.4cm);
		
		\node at(0,-2) {$D_1$};
	\end{scope}
		\begin{scope}[shift={(4,0)}]
		\draw (-1,1) circle (0.4cm);\draw (-1,-1) circle (0.4cm);
		\draw (0,1) circle (0.4cm);\draw (0,-1) circle (0.4cm);
		\draw (1,1) circle (0.4cm);\draw (1,-1) circle (0.4cm);
		\draw (-1.2,0.4) circle (0.1cm);\draw (-1.2,-0.4) circle (0.1cm);
		\draw (-0.8,0.4) circle (0.1cm);\draw (-0.8,-0.4) circle (0.1cm);
		\draw (-0.4,0.4) circle (0.1cm);\draw (-0.4,-0.4) circle (0.1cm);
		\draw (0.4,0.4) circle (0.1cm);\draw (0.4,-0.4) circle (0.1cm);
		\draw (0,0.4) circle (0.1cm);\draw (0,-0.4) circle (0.1cm);
		\draw (0.8,0.4) circle (0.1cm);\draw (0.8,-0.4) circle (0.1cm);
		\draw (1.2,0.4) circle (0.1cm);\draw (1.2,-0.4) circle (0.1cm);
		
	\node at(0,-2) {$D_2$};
	\end{scope}
	\begin{scope}[shift={(8,0)}]
		\draw (-1,1) circle (0.4cm);\draw (-1,-1) circle (0.4cm);
		\draw (0,1) circle (0.4cm);\draw (0,-1) circle (0.4cm);
		\draw (1,1) circle (0.4cm);\draw (1,-1) circle (0.4cm);
		\draw (-1.2,0.4) circle (0.1cm);\draw (-1.2,-0.4) circle (0.1cm);
		\draw (-0.8,0.4) circle (0.1cm);\draw (-0.8,-0.4) circle (0.1cm);
		\draw (-0.4,0.4) circle (0.1cm);\draw (-0.4,-0.4) circle (0.1cm);
		\draw (0.4,0.4) circle (0.1cm);\draw (0.4,-0.4) circle (0.1cm);
		\draw (0,0.4) circle (0.1cm);\draw (0,-0.4) circle (0.1cm);
		\draw (0.8,0.4) circle (0.1cm);\draw (0.8,-0.4) circle (0.1cm);
		\draw (1.2,0.4) circle (0.1cm);\draw (1.2,-0.4) circle (0.1cm);
		\draw (-1,0.15) circle (0.05cm);\draw (-0.8,0.15) circle (0.05cm);\draw (-0.6,0.15) circle (0.05cm);\draw (-0.4,0.15) circle (0.05cm);\draw (-0.2,0.15) circle (0.05cm);\draw (0,0.15) circle (0.05cm);\draw (0.2,0.15) circle (0.05cm);\draw (0.4,0.15) circle (0.05cm);\draw (0.6,0.15) circle (0.05cm);\draw (0.8,0.15) circle (0.05cm);\draw (1,0.15) circle (0.05cm);
		\draw (-1,-0.15) circle (0.05cm);\draw (-0.8,-0.15) circle (0.05cm);\draw (-0.6,-0.15) circle (0.05cm);\draw (-0.4,-0.15) circle (0.05cm);\draw (-0.2,-0.15) circle (0.05cm);\draw (0,-0.15) circle (0.05cm);\draw (0.2,-0.15) circle (0.05cm);\draw (0.4,-0.15) circle (0.05cm);\draw (0.6,-0.15) circle (0.05cm);\draw (0.8,-0.15) circle (0.05cm);\draw (1,-0.15) circle (0.05cm);
		\draw (-1,0)--(1,0); 
	\node at(0,-2) {$D$};	
	\end{scope}
	\end{tikzpicture}
	\caption{A sequence of circle domains converging to a domain that contains a line segment in the boundary and hence is not a circle domain.}\label{fig:caratheodory}
\end{figure}

Therefore, in order to ensure that the limit $D$ is a circle domain one needs to impose further requirements on the sequence $\{\Omega_n\}_{n\in \N}$ that approximates the original domain $\Omega$. The proofs of He--Schramm \cite{HeSchramm:Uniformization} and Schramm \cite{Schramm:transboundary} of Koebe's conjecture for countably connected domains proceed by \textit{externally} approximating the domain $\Omega$; that is, the approximating sequence $\{\Omega_n\}_{n\in \N}$ satisfies $\Omega_n\supset \Omega$ for each $n\in \N$. Even under this condition, the sequence of circle domains $\{D_n\}_{n\in \N}$ need not converge to a circle domain and one has to impose further conditions. For example, in the work of Karafyllia and the author for the proof of Theorem \ref{theorem:karafyllia_ntalampekos}, the approximating domains $\{\Omega_n\}_{n\in \N}$ satisfy a strong geometric condition, called \textit{inner uniformity}.

Recently, Rajala \cite{Rajala:koebe} approached the conjecture using instead \textit{internal} approximations, or else, {exhaustions} of $\Omega$. An  \textit{exhaustion}\index{exhaustion} of $\Omega$ is a sequence of domains $\{\Omega_n\}_{n\in \N}$, each bounded by finitely many Jordan curves that are contained in $\Omega$, such that
$$\Omega_n\subset \Omega_{n+1}\subset \Omega\,\,\, \text{for each $n\in \N$ and}\,\,\, \Omega=\bigcup_{n\in \N}\Omega_n.$$
Rajala gave an alternative proof of the He--Schramm uniformization theorem (Theorem \ref{theorem:he_schramm}) by using exhaustions. 

The method of using exhaustions appears in the problem of mapping conformally a domain onto a domain bounded by horizontal slits\index{domain!slit}\index{slit domain}, as in Theorem \ref{theorem:slit}. Specifically, if $\{\Omega_n\}_{n\in \N}$ is \textit{any} exhaustion of a domain $\Omega$ and, for each $n\in \N$, $f_n$ is a conformal map from $\Omega_n$ onto a domain bounded by finitely many horizontal slits, then after appropriate normalizations and passing to a subsequence, $\{f_n\}_{n\in \N}$ converges to a conformal map from $\Omega$ onto a horizontal slit domain \cite{Courant:Dirichlet}*{Theorem 2.1, p.~54}. 

In sharp contrast, when uniformizing by circle domains, Rajala observed that the exhaustions have to be chosen carefully, because the phenomenon of Figure \ref{fig:caratheodory} can appear in this setting as well. Generalizing the method of Rajala, the author and Rajala \cite{NtalampekosRajala:exhaustion} proved that the method of exhaustions can be used in all domains that satisfy Koebe's conjecture; in other words, if the conjecture is true, it can be proved by exhaustions.

\begin{theorem}[\cite{NtalampekosRajala:exhaustion}]\label{theorem:ntalampekos_rajala}
Let $\Omega\subset \widehat \C$ be a domain that is conformally equivalent to a circle domain. Then there exists an exhaustion $\{\Omega_n\}_{n\in \N}$ of $\Omega$ and for each $n\in \N$ a conformal map $f_n$ from $\Omega_n$ onto a finitely connected circle domain $D_n$ such that the sequence $\{f_n\}_{n\in \N}$ converges locally uniformly in $\Omega$ to a conformal map onto a circle domain.
\end{theorem}

Rajala \cite{Rajala:duality} studied other special conditions that must be true for domains that satisfy Koebe's conjecture. Therefore, in an attempt to disprove the conjecture, it suffices to find a domain that does not satisfy the conclusion of Theorem \ref{theorem:ntalampekos_rajala} or the conditions of Rajala \cite{Rajala:duality}.

%On the other hand, there are other potential approaches towards proving the conjecture beyond the approximation methods discussed above. Specifically, Schramm introduced a potential theoretic method for proving the conjecture in the countably connected case \cite{Schramm:transboundary_long}. Bonk \cite{Bonk:square} introduced an extremal problem on finitely connected domains whose solution gives a conformal map onto a \textit{square domain}\index{square domain}\index{domain!square}, i.e., a domain whose complementary components are squares or points. It would be interesting to extend this method to countably connected domains and beyond. Note that if a domain is conformally equivalent to a square domain, then it is also conformally equivalent to a circle domain, thanks to Schramm's cofat uniformization theorem (Theorem \ref{theorem:schramm_cofat}). In fact, an equivalent formulation of Koebe's conjecture, is that each domain is conformally equivalent to a square domain.

\section{Quasiconformal uniformization by Schottky sets}\label{section:quasiconformal_uniformization}

\subsection{Quasiconformal maps}\label{section:qc}
An orientation-preserving homeomorphism $f\colon U\to V$ between open subsets of $\C$ is \textit{quasiconformal}\index{quasiconformal map}\index{map!quasiconformal} if $f$ lies in the Sobolev space $W^{1,2}_{\loc}(U)$ and  there exists $K\geq 1$ such that
\begin{align}\label{definition:qc}
\|Df(z)\|^2\leq KJ_f(z)
\end{align}
for a.e.\ $z\in U$; here $\|Df(z)\|$ denotes the operator norm of the differential of $f$ at $z$ and $J_f$ is the Jacobian determinant of $Df$. In this case we say that $f$ is $K$-quasiconformal. This is known as the \textit{analytic}\index{quasiconformal map!analytic definition} definition of quasiconformality.

A homeomorphism $f\colon U\to V$ between open subsets of the Riemann sphere $\widehat \C$ is quasiconformal if $f|_{U\setminus \{\infty, f^{-1}(\infty\}\}}$ is quasiconformal in the above sense. 

There are several equivalent definitions of quasiconformal maps. According to the \textit{metric definition}\index{quasiconformal map!metric definition} of quasiconformality, an orientation-preserving homeomorphism $f\colon U\to V$ between open subsets of $\C$ is quasiconformal if for each $z\in U$ and for every sufficiently small $r>0$ (depending on $z$), there exists $R>0$ such that
$$B(f(z),R)\subset f(B(z,r))\subset B(f(z),HR),$$
where $H\geq 1$ is a uniform constant; see Figure \ref{fig:qc_metric}. Informally, $f$ maps infinitesimal balls to sets of bounded eccentricity. Compare this property to the corresponding property of conformal maps, which map infinitesimal balls to infinitesimal balls. Another geometric interpretation of \textit{smooth} quasiconformal maps is that they distort angles in a bounded way, as opposed to conformal maps, which preserve angles. Thus, we may regard quasiconformal maps as a more flexible class than conformal maps, but with some distortion control imposed.

\begin{figure}
\begin{tikzpicture}[scale=.45]
\draw [line width=.5pt,fill=black,fill opacity=0.1] (0.76,-0.2) circle (2.98cm);
\draw [fill=black] (0.76,-0.2) circle (1.5pt);
\draw[color=black] (0.9,0.2) node {$z$};
\draw(0.9,-4) node {$B(z,r)$};

\draw (6.2,1.5) node {$f$};
\draw[->] (5.2,0.4) to [out=30, in=150] (7.2,0.4);
\draw[line width=.5pt,fill=black,fill opacity=0.1, rounded corners=8] (10.74,2.4) -- (10.6,1.1) -- (9.04,0.46) -- (11.58,-2.42) -- (14.94,-0.9) -- (14.46,2.94) -- (12.78,2.06) -- cycle;
\draw [fill=black] (12.54,0.18) circle (1.5pt);
\draw[color=black] (12.7,0.7) node {$f(z)$};
\draw(12.54,-4) node {$B(f(z),HR)$};
\draw (12.54,0.18)--(14.5,0.18)node[pos=.5,below]{$R$};

\draw [line width=.5pt] (12.54,0.18) circle (1.95cm);
\draw [line width=.5pt] (12.54,0.18) circle (3.27cm);
\end{tikzpicture}
\caption{The metric definition of quasiconformality.}\label{fig:qc_metric}
\end{figure}
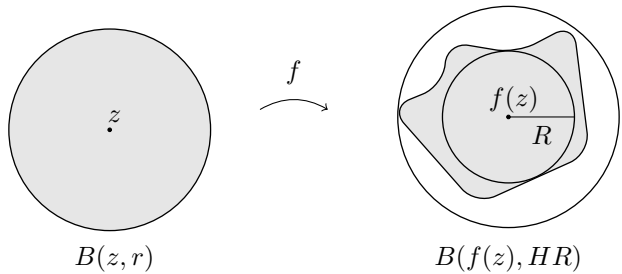

We include some fundamental properties of quasiconformal maps. Every $1$-quasiconformal map is conformal and vice versa. Note, though, that quasiconformal maps are not analytic or smooth in general. The inverse of a $K$-quasiconformal map is $K$-quasiconformal. Also, the composition of a $K_1$-quasiconformal map with a $K_2$-quasiconformal map is $K_1K_2$-quasiconformal. In particular, the composition of a $K$-quasiconformal map with a conformal map is $K$-quasiconformal. We direct the reader to \cites{Ahlfors:qc, Vaisala:quasiconformal, LehtoVirtanen:quasiconformal, AstalaIwaniecMartin:quasiconformal} for further background. 

In this section we work almost exclusively with quasiconformal maps of the entire sphere $\widehat \C$. Henceforth, we say that two sets $E,F$ are \textit{quasiconformally equivalent} if there exists a quasiconformal homeomorphism of $\widehat \C$ that maps $E$ onto $F$. Compare this with the notion of conformal equivalence of domains that we define in Section \ref{section:koebe}.

If $A$ is a positive parameter or collection of parameters, we use the notation $C(A)$ for a positive constant that depends only on $A$.

\subsection{Quasidisks}\label{section:quasidisks} 

A set $U\subset \widehat \C$ is a \textit{quasidisk}\index{quasidisk} if it is the image of the unit disk $\D$ (or equivalently of any other disk) under a quasiconformal map $f\colon \widehat \C \to \widehat \C$. If $f$ is $K$-quasiconformal for some $K\geq 1$, then $U$ is called a $K$-quasidisk. An intriguing problem since the early development of the theory of quasiconformal maps had been to provide a geometric characterization of quasidisks. This problem was resolved in a seminal work of Ahlfors \cite{Ahlfors:reflections}.

\begin{theorem}[\cite{Ahlfors:reflections}]\label{theorem:ahlfors_quasicircle}
A Jordan region $U\subset \C$ is a quasidisk if and only if $\partial U$ is a quasicircle, quantitatively. 
\end{theorem}

A Jordan curve $J\subset  \C$ is a \textit{quasicircle}\index{quasicircle} if there exists a constant $L\geq 1$ such that for every two points $z,w\in J$ there exists an arc $E\subset J$ connecting $z,w$ such that 
$$\diam E \leq L|z-w|.$$
In this case we say that $J$ is an $L$-quasicircle. One may define quasicircles in the sphere $\widehat \C$ by using the spherical metric in the above definition (see the discussion in \cite{Ntalampekos:tangent}*{Section 2.4}); Ahlfors' theorem remains true in this setting. The statement of Theorem \ref{theorem:ahlfors_quasicircle} is \textit{quantitative} in the sense that if $U$ is a $K$-quasidisk, then $\partial U$ is an $L$-quasicircle for some $L$ that depends only on $K$, and vice versa.

Geometrically, quasicircles do not have cusps. It is straightforward to establish that several fractal sets, especially self-similar ones, are quasicircles; hence, by the theorem of Ahlfors, they can be mapped to the unit circle via a quasiconformal map. See Figure \ref{fig:quasicircle} for examples of quasicircles and Figure \ref{fig:quasicircle_non} for a curve that is not a quasicircle.

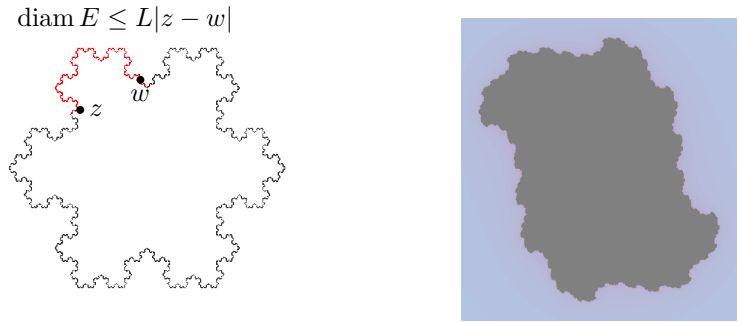
\begin{figure}
	\centering
	\begin{tikzpicture}
	\node at (0,0) {\includegraphics[scale=0.4]{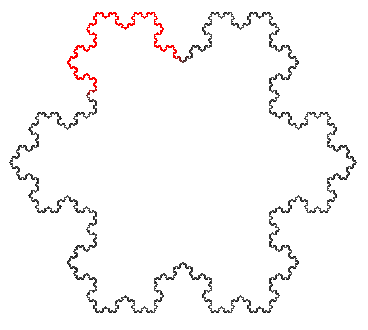}};
	\fill (-0.9,0.8) circle (1.5pt) node [right] {$z$};
	\fill (-0.1,1.2) circle (1.5pt) node [below] {$w$}; 
	\node at (-0.3,2) {$\diam E \leq L|z-w|$};
		
	\node[opacity=0.5] at (6,0) {\includegraphics[scale=0.3]{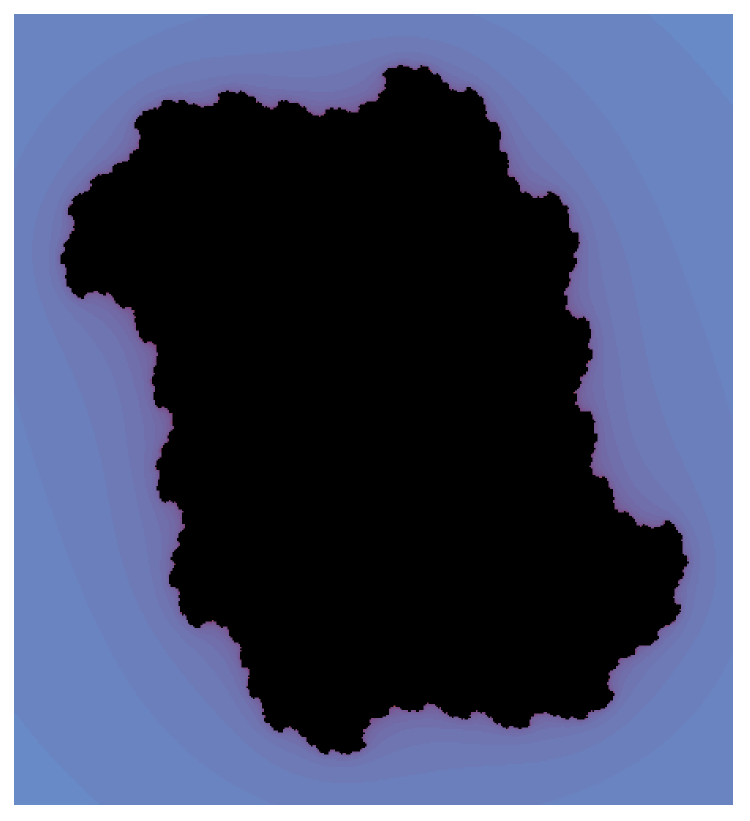}};
	\end{tikzpicture}
	\caption{Left: The von Koch snowflake is a quasicircle. Right: The Julia set of $z^2+c$ for $c\in \C$ close to $0$ is a quasicircle.}\label{fig:quasicircle}
\end{figure}

\begin{figure}
	\centering
	\begin{tikzpicture}
	\node at (0,0) {\includegraphics[scale=0.3]{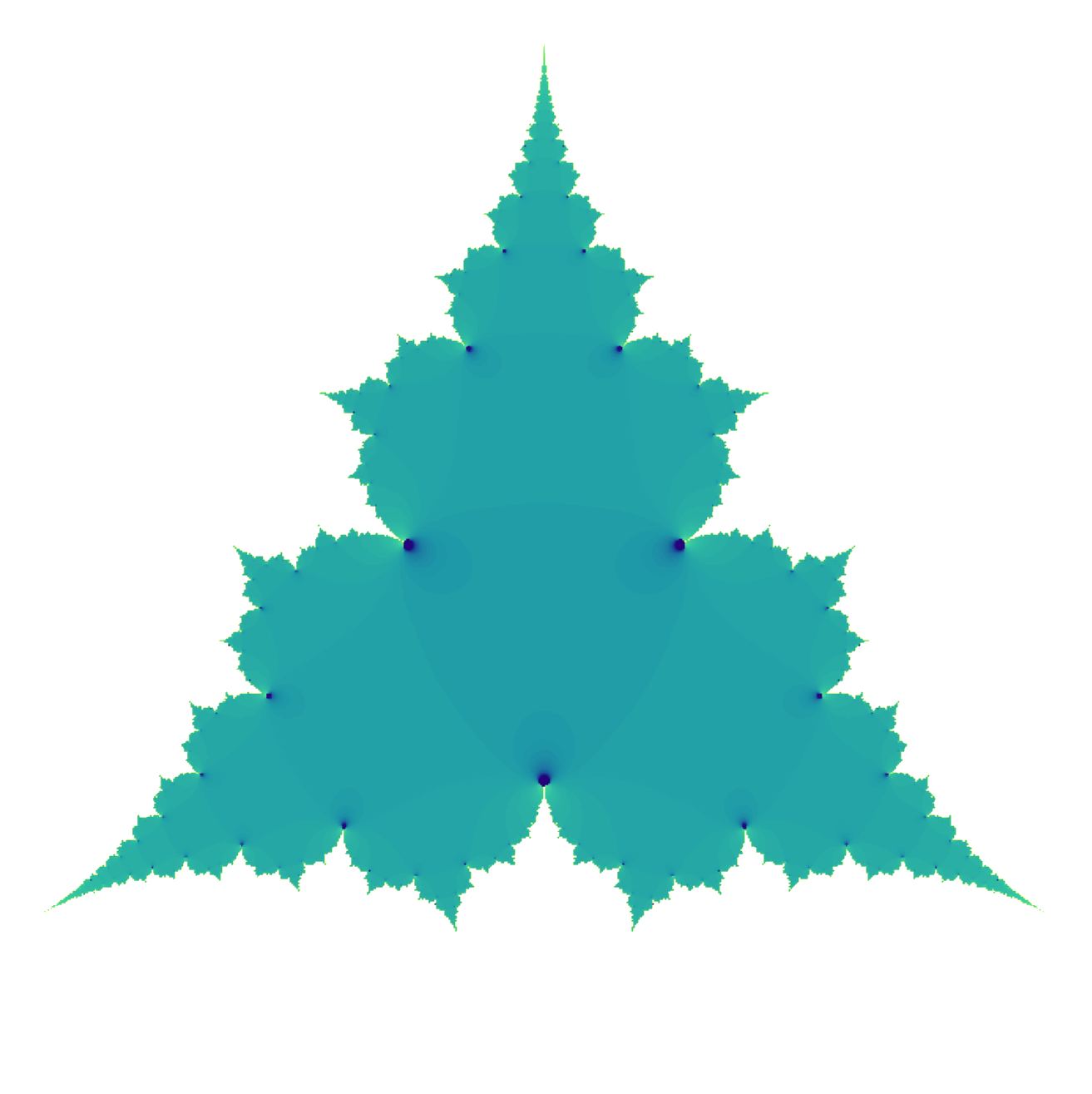}};
	\fill (-0.3,0.9) circle (1.5pt) node[below] {$z$};
	\fill (0.3,0.9) circle (1.5pt) node[below] {$w$};
	\draw[<->] (0.8,0.9) -- (0.8,2.3) node[pos=.5, right] {$\diam E \gg |z-w|$};	
	\end{tikzpicture}
	\caption{Quasicircles do not have cusps. The curve shown is the Julia set of a rational map that has cusps, so it is not a quasicircle.}\label{fig:quasicircle_non}
\end{figure}
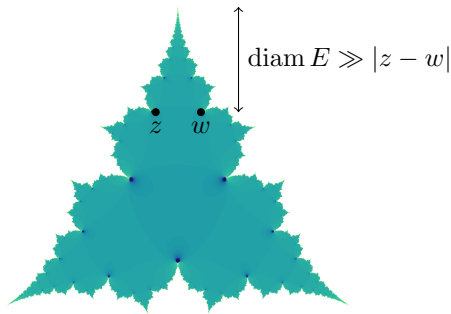

\subsection{Quasiconformal annuli}

Thanks to Theorem \ref{theorem:ahlfors_quasicircle}, we have a complete understanding of the quasiconformal uniformization problem for Jordan curves. It is natural to ask whether a pair of disjoint Jordan curves can be mapped quasiconformally to a pair of circles in a controlled manner. Namely, if $J_1,J_2$ are Jordan curves, we wish to map them to circles with a $K$-quasiconformal homeomorphism of the sphere so that $K$ depends only on the geometric features of $J_1$ and $J_2$. By Theorem \ref{theorem:ahlfors_quasicircle} it is necessary that $J_1$ and $J_2$ are quasicircles, quantitatively depending on $K$. Is this condition sufficient?

As a motivating example, note that each open square in the plane is a $L$-quasidisk. Suppose that any pair of disjoint squares can be mapped to a pair of disjoint disks with a $K$-quasiconformal map of the sphere, for some uniform $K\geq 1$. Consider the unit square $[0,1]^2$ and for each $n\in \N$ let $A_n$ be a square with sides parallel to the axes such that $A_n$ is larger than $[0,1]^2$ and the distance of a vertex of $A_n$ from a vertex of $[0,1]^2$ tends to $0$ as $n\to\infty$; see Figure \ref{fig:square}. For each $n\in \N$ consider a $K$-quasiconformal map $f_n\colon \widehat \C\to \widehat \C$ that maps $[0,1]^2$ to the unit disk and the square $A_n$ to a disk. By normality criteria for quasiconformal maps (see \cite{LehtoVirtanen:quasiconformal}*{Section II.5}), after appropriate normalizations, a subsequence of $f_n$ converges to a quasiconformal map $f$ that maps the unit square onto $\D$ and a square or quarter plane $A$ to a disk or half plane $f(A)$. Note that $A$ and $[0,1]^2$ have a common vertex, while $f(A)$ is tangent to $\D$. Since quasiconformal maps quasi-preserve angles, the map $f$ cannot map the non-zero angle between $A$ and $[0,1]^2$ to the zero degree angle between $f(A)$ and $\D$; one can prove this formally using the fact that quasiconformal maps are locally quasisymmetric. Hence, we are lead to a contradiction. See \cite{Ntalampekos:CarpetsThesis}*{Section 3.12} for further details.

\begin{figure}
	\begin{tikzpicture}
	\begin{scope}
	\draw[shift={(-2.1,1.1)},scale=2,fill=yellow!50] (0,0)--(1,0)--(1,1)--(0,1)--cycle; 
	\node at (-0.55,1.5) {$A_n$};
	\draw[fill=black!20] (0,0)--(1,0)--(1,1)--(0,1)--cycle;
	\node at (0.5,0.5) {$[0,1]^2$};
	\draw[->] (1.5,1)--(3.5,1) node[above, pos=.5]{$L$-quasiconformal};
	\draw[fill=yellow!50] (5,1) circle (.9cm);
	\draw[fill=black!20] (6.5,1) circle (0.4cm);
	\node at (6.5,1) {$\D$};
	\end{scope}

	\begin{scope}[shift={(0,-4)}]
	\fill[color=yellow!50](-2,1)--(0,1)--(0,3)--(-2,3)--cycle;\draw (-2,1)--(0,1)--(0,3);
	\node at (-0.55,1.5) {$A$};
	\draw[fill=black!20] (0,0)--(1,0)--(1,1)--(0,1)--cycle;
	\node at (0.5,0.5) {$[0,1]^2$};
	\draw[->] (1.5,1)--(3.5,1) node[above, pos=.5] {quasiconformal};
	\node at (2.5,1.2){\textcolor{red}{\Huge{\xmark}}};
	\draw[fill=black!20] (6,1) circle (0.5cm) node {$\D$};	
	\fill[color=yellow!50] (5.5,0)--(5.5,2)--(4.5,2)--(4.5,0)--cycle;
	\draw (5.5,0)--(5.5,2);
	\end{scope}
\end{tikzpicture}
	\caption{}\label{fig:square}
\end{figure}

The above example shows that the good geometry of squares is not sufficient to guarantee that a pair of squares can be mapped to a pair of disks with controlled distortion that is independent of other features, such as the \textit{relative distance} of the squares. The {relative distance}\index{relative distance} of two sets $E,F\subset \C$ with positive and finite diameters is defined as
$$\Delta(E,F) =\frac{\dist(E,F)}{\min\{\diam E,\diam F\}}.$$

Herron observed that if one controls the relative distance of two quasidisks, then it is possible to map them to disks in a controlled way. Specifically, the following statement is a consequence of \cite{Herron:uniform}*{Theorem 2.6 and Corollary 3.5}.

\begin{theorem}[\cite{Herron:uniform}]\label{theorem:herron}
Let $L\geq 1$, $\delta>0$, and $U,V\subset \C$ be a pair of $L$-quasidisks such that $\Delta(U,V)\geq \delta$. Then there exists a $K(L,\delta)$-quasiconformal homeomorphism of $\widehat \C$ that maps $U$ and $V$ to disks. 
\end{theorem}

Note that the converse is not true. Namely, if we can map $U$ and $V$ quasiconformally to disks, then $U$ and $V$ need not have large relative distance. For instance, this is the case if $U$ and $V$ are already disks that are too close to each other. We are naturally led to the following problem.

\begin{problem}\label{problem:annuli}
Find a necessary and sufficient condition so that a pair of disjoint quasidisks can be quasiconformally mapped to a pair of disjoint disks in a quantitative way. 
\end{problem}

We will return back to this problem in Section \ref{section:annuli_revisit}.

\subsection{Schottky sets}
More generally, one can ask whether three or more disjoint quasidisks can be quasiconformally mapped to disks. The complement of a collection of disjoint disks in $\widehat \C$ is called a \textit{Schottky set}\index{Schottky set}; see Figure \ref{fig:schottky}. Schottky sets appear in several problems in geometric function theory such as in Koebe's conjecture, since the closure of a circle domain is a Schottky set. Moreover, they appear as limit sets of Kleinian groups, so they are also studied from a dynamical point of view.

\begin{figure}
	\centering
	\includegraphics[scale=0.5]{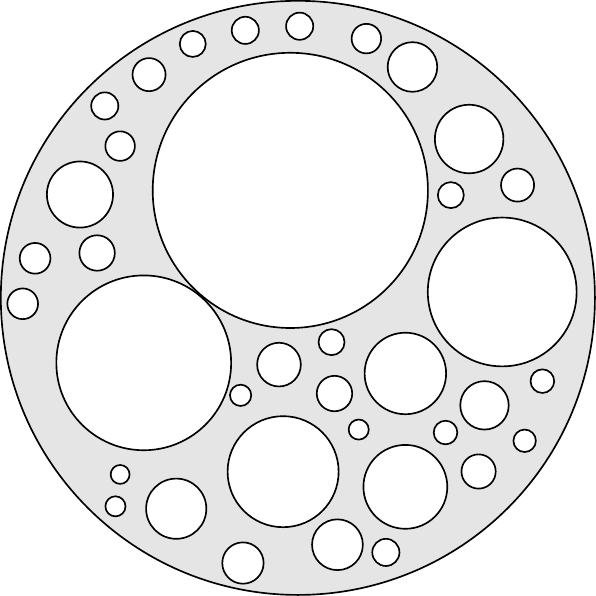}\hfill
	\includegraphics[scale=0.5]{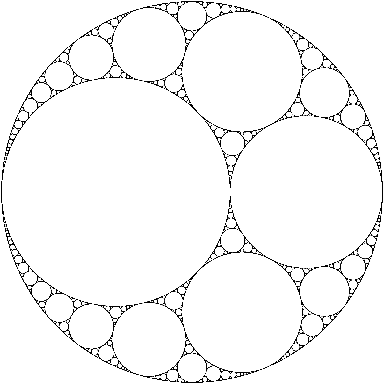}
	\caption{Examples of Schottky sets. The Schottky set in the right is known as the Apollonian gasket\index{gasket!Apollonian}\index{Apollonian gasket}.}\label{fig:schottky}
\end{figure}

\subsubsection{Early results}\label{section:early}
The problem of quasiconformal uniformization by Schottky sets dates back to 1990, when Herron and Koskela \cite{HerronKoskela:QEDcircledomains} proved Koebe's conjecture for uniform domains. In the same work they also obtained the following result; compare with Theorem \ref{theorem:herron_koskela}.

\begin{theorem}[\cite{HerronKoskela:QEDcircledomains}]\label{theorem:herron_koskela_qc}
Let $\Omega\subset \C$ be a uniform domain\index{uniform domain}\index{domain!uniform}. Then there exists a quasiconformal homeomorphism of $\widehat \C$ that maps $\bar \Omega$ onto a Schottky set.
\end{theorem}

At the same time McMullen \cite{McMullen:teichmuller} proved an analogous result for limit sets of Kleinian groups.

\begin{theorem}[\cite{McMullen:teichmuller}]\label{theorem:mcmullen}
If the limit set of a convex cocompact Kleinian group\index{Kleinian group} is a Sierpi\'nski carpet, then there exists a quasiconformal homeomorphism of $\widehat\C$ that maps it onto a Schottky set.
\end{theorem}

Among the notions appearing in this result, we restrict ourselves to defining Sierpi\'nski carpets. The \textit{standard Sierpi\'nski carpet}\index{standard Sierpi\'nski carpet}\index{Sierpi\'nski carpet!standard} is a self-similar fractal set that is constructed by subdividing the unit square $[0,1]^2$ into nine squares of side length $1/3$ and removing the middle square; then we proceed inductively in each of the remaining eight squares. More generally, a Sierpi\'nski carpet is defined as follows.

\begin{definition}[Sierpi\'nski carpet]\index{Sierpi\'nski carpet}
A set $S\subset \widehat \C$ is a \textit{Sierpi\'nski carpet} if it satisfies the following conditions:
\begin{enumerate}[label=\normalfont(\arabic*)]
\item\label{carpet:1} The complementary components of $S$ are countably many disjoint Jordan regions $U_i$, $i\in \N$.
\item\label{carpet:2} $S$ has empty interior.
\item\label{carpet:3} $\diam U_i\to 0$ as $i\to\infty$ (measured in the spherical metric).
\item\label{carpet:4} $\bar{U_i}\cap \bar{U_j}=\emptyset$ for $i\neq j$.
\end{enumerate}
\end{definition}

It is a fundamental result of Whyburn \cite{Whyburn:theorem} that all Sierpi\'nski carpets are homeomorphic to each other and, in particular, to the standard Sierpi\'nski carpet. See Figure \ref{fig:carpet} for an illustration. 

\begin{figure}
	\includegraphics[scale=0.5]{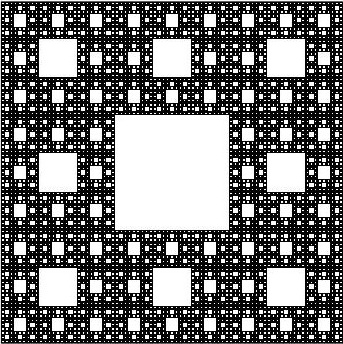}\hfill
	\includegraphics[scale=0.37]{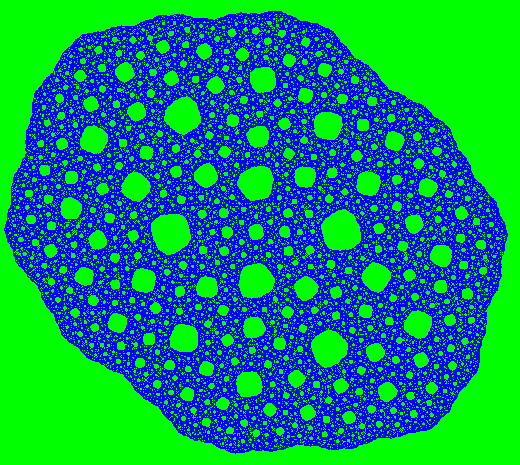}
	\caption{The standard Sierpi\'nski carpet and a Sierpi\'nski carpet Julia set of a rational map.}\label{fig:carpet}
\end{figure}

\subsubsection{Bonk's uniformization theorem}
Two decades after the results of Herron--Koskela and McMullen, Bonk proved in \cite{Bonk:uniformization} a general sufficiency criterion for quasiconformal uniformization by Schottky sets that does not rely on the geometry of uniform domains or on complex dynamics.

\begin{theorem}[\cite{Bonk:uniformization}]\label{theorem:bonk}
Let $\{U_i\}_{i\in I}$ be a collection of disjoint Jordan regions in $\widehat \C$. Suppose that there exists $L\geq 1$ such that 
\begin{enumerate}[label=\normalfont(\arabic*)]
\item\label{theorem:bonk:1} $U_i$ is an $L$-quasidisk for each $i\in I$ and
\item\label{theorem:bonk:2} for every pair of distinct indices $i,j\in I$ we have $\Delta(U_i,U_j)\geq L^{-1}$.
\end{enumerate}
Then there exists a $K(L)$-quasiconformal homeomorphism $f$ of the sphere that maps the set $S=\widehat \C\setminus \bigcup_{i\in I}U_i$ onto a Schottky set. Moreover, if $S$ has area zero, then $f$ is unique up to postcomposition with M\"obius transformations.
\end{theorem} 

Equivalently, we may rephrase the conclusion by saying that $f$ maps each $U_i$ to a disk. Note that the relative distance $\Delta(U_i,U_j)$ in \ref{theorem:bonk:2} has to be computed using the spherical metric. Actually, the metric is not so important because the condition that $\Delta(U_i,U_j)$ is bounded from below, as in \ref{theorem:bonk:2}, holds with the spherical metric if and only if it holds with the Euclidean metric; see \cite{Ntalampekos:tangent}*{Remark 2.3}. Also, in the case that $S$ has area zero, the set $S$ is actually a Sierpi\'nski carpet. 

The result of Bonk generalizes the idea of Herron in Theorem \ref{theorem:herron}, which uniformizes of a pair of Jordan regions by a pair of disks. It is remarkable that the exact same sufficient conditions for uniformizing a pair of Jordan regions apply to the case of infinitely many ones. 

We note that Bonk's theorem implies the results of Herron--Koskela (Theorem \ref{theorem:herron_koskela_qc}) and McMullen (Theorem \ref{theorem:mcmullen}) mentioned above. One may ask the reason for a gap of two decades between these results and Bonk's result. The answer is that in the meantime Schramm \cite{Schramm:transboundary} devised the powerful tool of transboundary modulus\index{transboundary modulus} for the study of uniformization problems in the plane. Bonk relied crucially on this tool and studied it systematically, laying the groundwork for several further projects involving Sierpi\'nski carpets and domains in the plane and in metric spaces. 

Bonk’s remarkable criterion has been used to address rigidity problems arising in the study of the standard Sierpiński carpet and Julia sets \cites{BonkMerenkov:rigidity, BonkLyubichMerenkov:carpetJulia, BonkMerenkov:rigiditySpCarpets}, and it has been further generalized to the setting of non-planar carpets \cites{MerenkovWildrick:uniformization, Rehmert:thesis}. We state the main result of the work of Bonk and Merenkov \cite{BonkMerenkov:rigidity}, which represents a major achievement in this area, and its proof illustrates how Schottky sets can be used to address problems not directly related to Schottky sets or circles.
\begin{theorem}[\cite{BonkMerenkov:rigidity}]\label{theorem:bonk_merenkov}
Every quasisymmetric homeomorphism of the standard Sierpi\'nski carpet is a Euclidean isometry.
\end{theorem}
One of the central problems in geometric group theory is to identity sets that can arise as limit sets of hyperbolic groups. A deep consequence of Theorem \ref{theorem:bonk_merenkov} is that the standard Sierpi\'nski carpet is not quasisymmetric to the limit set of any hyperbolic group. Indeed, the carpets arising as limit sets of hyperbolic groups admit infinite groups of quasisymmetric self-maps acting on them, whereas the corresponding group for the standard Sierpi\'nski carpet is finite by Theorem \ref{theorem:bonk_merenkov}. See also the relevant discussion below in Section \ref{section:gasket}.

\subsubsection{Relaxing the geometric assumptions}
The uniqueness of the map $f$ in the last sentence of the statement of Theorem \ref{theorem:bonk} follows from a work of Bonk--Kleiner--Merenkov \cite{BonkKleinerMerenkov:schottky} on the rigidity of Schottky sets. The geometry of circles is vitally used in this result and its proof involves the group generated by conformal reflections along the circles of a Schottky set. On the other hand, the existence part of Bonk's theorem does not rely on the geometry of circles. In fact, the proof can be adapted to quasiconformally uniformize the set $S$ by a set whose complementary components are squares or equilateral triangles. 

Now, regarding the necessity of Bonk's assumptions, it is obvious that \ref{theorem:bonk:1} is a necessary assumption, as in the case of Herron's theorem (Theorem \ref{theorem:herron}). However, condition \ref{theorem:bonk:2} is certainly not necessary, as there exist Schottky sets that do not satisfy it. 

It is natural to ask whether a uniformization result is true if one relaxes assumption \ref{theorem:bonk:2}. This problem is studied by the author in \cite{Ntalampekos:CarpetsThesis}*{Section 3}. It is shown that it is possible to uniformize a Sierpi\'nski carpet of area zero that satisfies a relaxed version of assumption \ref{theorem:bonk:1} by a square carpet via a homeomorphism that is quasiconformal in a generalized sense. Informally, if one relaxes the assumptions of Theorem \ref{theorem:bonk}, then a uniformizing map still exists but its regularity deteriorates. 

We present another result of the author \cite{Ntalampekos:uniformization_packing} in this direction. Quite surprisingly, we can remove entirely the geometric assumptions of Theorem \ref{theorem:bonk} and use instead a very weak summability condition, while still obtaining a uniformization result.

\begin{theorem}[\cite{Ntalampekos:uniformization_packing}]\label{theorem:packing_qc}
Let $\{U_i\}_{i\in I}$ be a collection of Jordan regions in $\widehat \C$ that have disjoint closures and satisfy
$$\sum_{i\in I} (\diam U_i)^2<\infty.$$
Then there exists a packing-quasiconformal map $f\colon \widehat \C\to \widehat \C$ that maps a Schottky set onto $S=\widehat{\C}\setminus \bigcup_{i\in I}U_i$. 
\end{theorem}

We refrain from defining the notion of a packing-quasiconformal map\index{map!packing-quasiconformal}\index{packing-quasiconformal map} here. Essentially, as in the definition of classical quasiconformal maps in Section \ref{section:qc}, $f|_{f^{-1}(S)}$ is required to lie in an appropriate generalization of Sobolev spaces and to satisfy a distortion inequality analogous to \eqref{definition:qc}. No regularity is imposed in the preimages $f^{-1}(U_i)$. Moreover, $f$ is not necessarily a homeomorphism, but it is only a uniform limit of homeomorphisms; equivalently, $f$ is continuous, surjective, and the preimage of each point is connected \cite{Youngs:monotone}. The definition of packing-quasiconformal maps is natural and is motivated by various definitions of Sobolev spaces and quasiconformal maps in metric spaces that have been studied in the past three decades \cites{HeinonenKoskela:qc, Shanmugalingam:newtonian, Heinonen:metric, Williams:qc, HeinonenKoskelaShanmugalingamTyson:Sobolev, Ntalampekos:CarpetsThesis, NtalampekosRomney:length, NtalampekosRomney:nonlength}.

The motivation between the summability condition of Theorem \ref{theorem:packing_qc} originates from the study of the \textit{conformal loop ensemble carpet}\index{conformal loop ensemble}\index{CLE} (CLE), introduced by Sheffield and Werner \cite{SheffieldWerner:CLE}. This is a Sierpi\'nski carpet that arises by a random collection of Jordan curves in the unit disk that combines conformal invariance and a natural restriction property. The geometry of this random carpet is singular and its complementary components are neither quasidisks, nor do they satisfy the separation condition \ref{theorem:bonk:2}. However, Rohde and Werness \cite{RohdeWerness:CLE} announced that with probability $1$ the diameters satisfy the summability condition of Theorem \ref{theorem:packing_qc}. The first publicly available proof of this announced result is due to Doherty and Miller \cite{DohertyMiller:CLE_square}. Therefore, by Theorem \ref{theorem:packing_qc}, CLE carpets can be uniformized by Schottky sets with probability $1$; see Figure \ref{fig:cle}.

\begin{figure}
	\centering
	\begin{tikzpicture}
\node at (0,0) {\includegraphics[scale=.25]{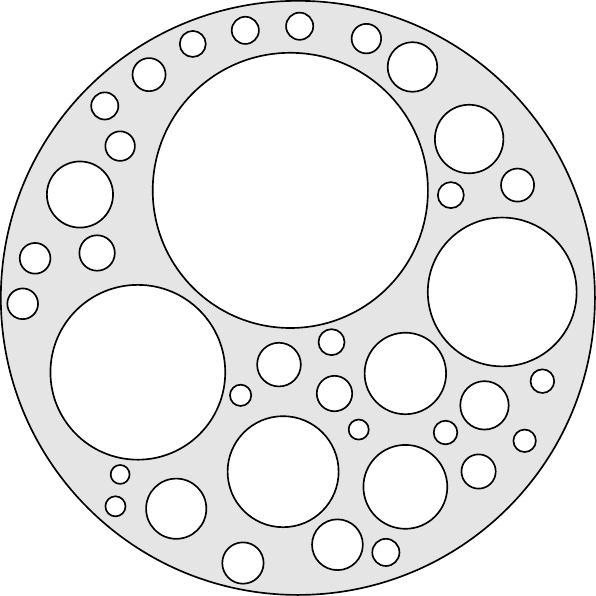}};
\draw[->] (2,0)--(3,0);
\node at (5,0) {\includegraphics[scale=0.07]{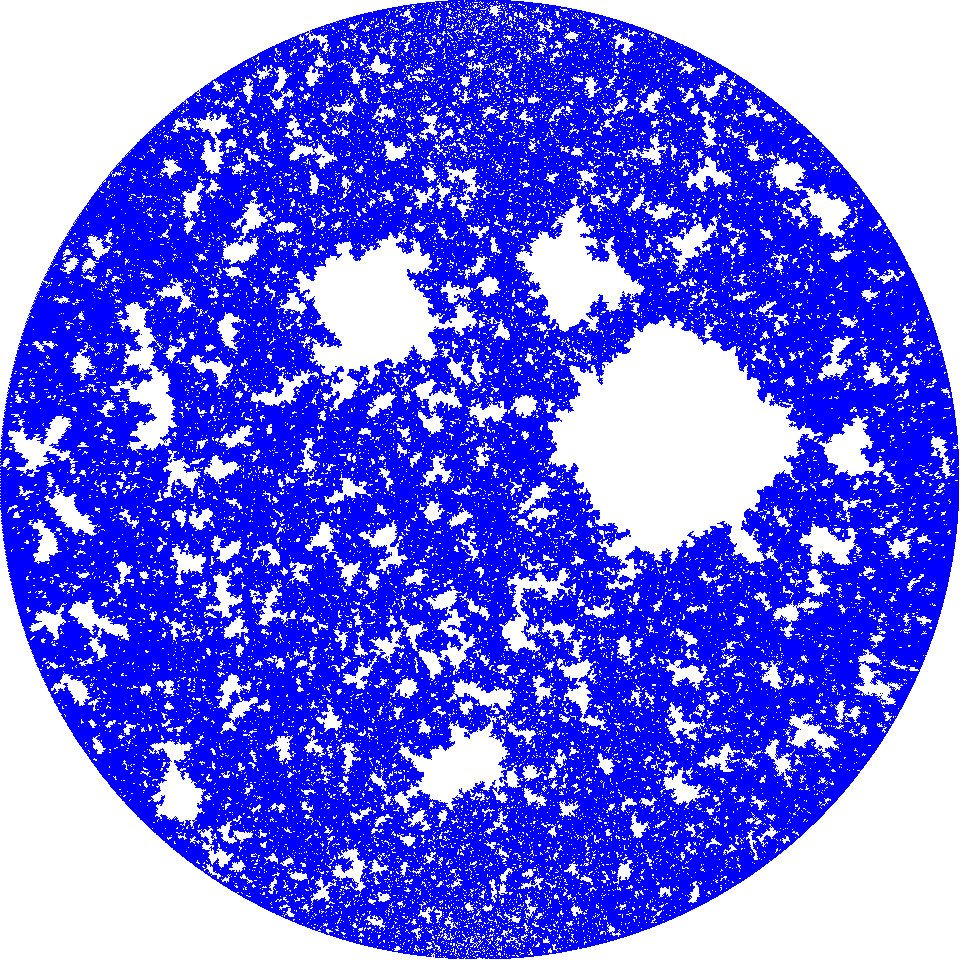}};
\end{tikzpicture}
\caption{Uniformization of a CLE carpet by a Schottky set. The simulation of the CLE carpet shown is due to D.\ Wilson.}\label{fig:cle}
\end{figure}

\subsubsection{Uniformization of Julia sets}\label{section:gasket}

We return to the problem of uniformizing sets by Schottky sets using quasiconformal maps in the classical sense. 

Bonk's theorem is amenable to uniformization results in complex dynamics because the two main assumptions in Theorem \ref{theorem:bonk} are very natural in the setting of (sub)hyperbolic dynamical systems and can be readily verified. In this direction, Bonk, Lyubich and Merenkov \cite{BonkLyubichMerenkov:carpetJulia} proved the following result.

\begin{theorem}[\cite{BonkLyubichMerenkov:carpetJulia}]\label{theorem:bonk_lyubich_merenkov}\index{Julia set}
Let $R$ be a postcritically finite rational map whose Julia set $\mathcal J(R)$ is a Sierpi\'nski carpet. Then there exists a quasiconformal homeomorphism of $\widehat \C$ that maps $\mathcal J(R)$ onto a Schottky set. 
\end{theorem}

We direct the reader to the classical treatise \cite{Milnor:dynamics} for definitions of dynamical notions. See Figure \ref{fig:carpet} for a Sierpi\'nski carpet Julia set satisfying the assumptions of Bonk's theorem. The above theorem has some deep consequences in the rigidity of such Julia sets. For example, it is proved in \cite{BonkLyubichMerenkov:carpetJulia} that the group of quasisymmetric self-maps of $\mathcal J(R)$ is finite, and therefore, no limit set of Kleinian group can be quasisymmetrically equivalent to $\mathcal J(R)$. Extensions of these results have recently appeared in \cite{MerenkovShen:quasiregular}.

While Bonk's criterion in Theorem \ref{theorem:bonk} can be applied in several interesting and important cases, it does not allow the complementary quasidisks $U_i$, $i\in I$, to be close to each other, let alone to touch each other. For this reason a general uniformization result for sets that topologically resemble the Apollonian gasket (see Figure \ref{fig:schottky}) seemed to be out of reach, even when the complementary components $U_i$, $i\in I$, satisfy very strong geometric conditions.

Recently the author and Luo \cite{LuoNtalampekos:gasket} were able to overcome this obstacle and prove a quasiconformal uniformization result for gasket Julia sets. Motivated by the Apollonian gasket, we  give the following definition.

\begin{definition}[Gasket]\label{definition:gasket}\index{gasket}
A set $K\subset \widehat \C$ is a \textit{gasket} if it satisfies the following conditions:
\begin{enumerate}[label=\normalfont(\arabic*)]
\item\label{gasket:1} The complementary components of $K$ are countably many disjoint Jordan regions $U_i$, $i\in \N$.
\item\label{gasket:2} $K$ has empty interior.
\item\label{gasket:3} $\diam U_i\to 0$ as $i\to\infty$ (measured in the spherical metric).
\item\label{gasket:4} The boundaries of two complementary components of $K$ share at most one point. 
\item\label{gasket:5} No point of $\widehat \C$ belongs to the boundaries of three complementary components of $K$. 
\item\label{gasket:6} The contact graph\index{contact graph} corresponding to $K$, obtained by assigning a vertex to each complementary component and an edge if two components share a boundary point, is connected. 
\end{enumerate}
\end{definition}

Note that conditions \ref{gasket:1}, \ref{gasket:2}, and \ref{gasket:3} are also satisfied by Sierpi\'nski carpets. However, in the case of carpets the regions $U_i$, $i\in \N$, have disjoint closures so \ref{gasket:4} and \ref{gasket:5} are trivially satisfied, while \ref{gasket:6} fails since the contact graph has no edges. Moreover,  Schottky sets satisfy necessarily conditions \ref{gasket:4} and \ref{gasket:5} so any set that can be mapped to a Schottky set with a homeomorphism of the sphere must also satisfy these two conditions.

We say that a gasket $K=\widehat \C\setminus \bigcup_{i=1}^\infty U_i$ is \textit{fat}\index{fat gasket}\index{gasket!fat} if $U_i$ and $U_j$ are tangent to each other whenever $\bar{U_i}\cap \bar {U_j}=\emptyset$; see Figure \ref{fig:julia_gasket}. We direct the reader to \cite{LuoNtalampekos:gasket} for the precise definition of fatness. Note that fatness here is not to be confused with the notion used earlier in Section \ref{section:geometric}; see also Problem \ref{problem:schottky_topological} below for a further discussion.

We state the main theorem of \cite{LuoNtalampekos:gasket}, which characterizes gasket Julia sets that can be quasiconformally mapped to Schottky sets. 

\begin{theorem}[\cite{LuoNtalampekos:gasket}]\label{theorem:luo_ntalampekos}\index{Julia set}
Let $R$ be a rational map whose Julia set $\mathcal J(R)$ is a gasket that does not contain any critical points of $R$. The following are equivalent.
\begin{enumerate}[label=\normalfont(\arabic*)]
\item There exists a quasiconformal homeomorphism $\phi$ of $\widehat \C$ that maps $\mathcal J(R)$ onto a Schottky set.
\item $\mathcal J(R)$ is a fat gasket.
\item Every contact point is eventually mapped to a parabolic periodic point with multiplicity $3$.
\end{enumerate}
Moreover, the map $\phi$ in (1) is unique in the following strong sense: if $\psi$ is any orientation-preserving homeomorphism of $\widehat \C$ that maps $\mathcal J(R)$ onto a Schottky set, then $\psi|_{\mathcal J(R)}$ agrees with $\phi|_{\J(R)}$ up to postcomposition with a M\"obius transformation.
\end{theorem}

\begin{figure}
	\centering
	\includegraphics[scale=0.3]{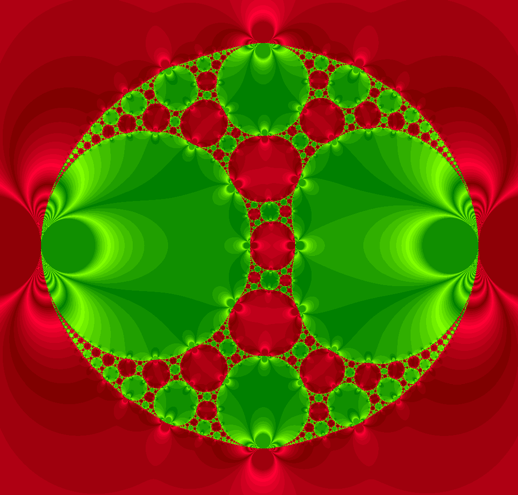}
	\caption{A fat gasket Julia set.}\label{fig:julia_gasket}
\end{figure}

We direct the reader to the referenced paper for the definitions of the various complex dynamical notions. This result relies heavily on the structure of Julia sets and on methods from complex dynamics, which do not generalize to arbitrary gaskets. Nonetheless, it serves as a motivation for the general result presented in the next section. 

\subsection{Characterization of Schottky sets}
Let $\{U_i\}_{i\in I}$ be a collection of disjoint Jordan regions in $\widehat \C$. We observe that in Bonk's theorem (Theorem \ref{theorem:bonk}) one imposes geometric assumptions on pairs of regions $U_i,U_j$ in order to obtain a quasiconformal uniformization result. Although this is less obvious, a similar philosophy applies to Theorem \ref{theorem:luo_ntalampekos}, where a tangency condition is imposed between pairs of regions $U_i,U_j$ that touch each other.  Motivated by these results we give the following definition.

\begin{definition}[Circularizable regions]\label{definition:circularizable}
Let $\{U_i\}_{i\in I}$ be a collection of disjoint Jordan regions in $\widehat \C$. We say that the regions $U_i$, $i\in I$, are \textit{uniformly pairwise quasiconformally circularizable}\index{quasiconformally circularizable regions} if there exists $K\geq1$ such that for every  $i,j\in I$ there exists a $K$-quasiconformal homeomorphism of $\widehat \C$ that maps $U_i$ and $U_j$ to disks with distinct boundaries.
\end{definition}

Note that if the regions $U_i$, $i\in I$, are $L$-quasidisks and $\Delta(U_i,U_j)\geq L^{-1}$ for $i\neq j$, then they are also uniformly pairwise quasiconformally circularizable by Herron's theorem (Theorem \ref{theorem:herron}). Thus, the condition of Definition \ref{definition:circularizable} is strictly weaker than the assumptions in Bonk's theorem (Theorem \ref{theorem:bonk}). 

It is natural to ask whether the above pairwise circularization condition is sufficient to imply quasiconformal uniformization by a Schottky set. The next result of the author \cite{Ntalampekos:schottky} provides an affirmative answer.

\begin{theorem}[\cite{Ntalampekos:schottky}]\label{theorem:ntalampekos_schottky}
Let $\{U_i\}_{i\in I}$ be a collection of disjoint Jordan regions in $\widehat \C$ that have distinct boundaries. The following are quantitatively equivalent.
\begin{enumerate}[label=\normalfont(\arabic*)]
\item The regions $U_i$, $i\in I$, are uniformly pairwise quasiconformally circularizable. 
\item There exists a quasiconformal homeomorphism of $\widehat \C$ that maps the set $S=\widehat \C\setminus \bigcup_{i\in I}U_i$ onto a Schottky set and is $1$-quasiconformal on $S$. 
\end{enumerate}
In this case the map $f|_S$ is unique up to postcomposition with M\"obius transformations.
\end{theorem}

This theorem provides a quasiconformal characterization of Schottky sets and yields Bonk's theorem as a corollary. The proof relies on the transboundary modulus of Schramm, on improving transboundary modulus estimates of Bonk, and on properties of groups generated by reflections on circles, known as Schottky groups. 

One of the difficulties in the proof is that there is no known \textit{topological} characterization of Schottky sets. We list three topological consequences of the circularization assumption of Definition \ref{definition:circularizable}; compare to the definitions of a Sierpi\'nski carpet in Section \ref{section:early} and of a gasket in Section \ref{section:gasket}.
\begin{enumerate}[label=\normalfont(\arabic*)]
\item $\bar {U_i}\cap \bar {U_j}$ contains at most one point for $i\neq j$. This is because there exists a homeomorphism of $\widehat \C$ mapping $U_i$ and $U_j$ to disks with distinct boundaries.  
\item $\bar {U_i}\cap \bar {U_j}\cap \bar {U_k}=\emptyset$ for every triple of distinct indices $i,j,k\in I$. This is more subtle and follows form the geometry of quasicircles and the characterization of Ahlfors in Theorem \ref{theorem:ahlfors_quasicircle}; see \cite{Ntalampekos:schottky}*{Lemma 3.2}.
\item For each $\varepsilon>0$ there exist at most finitely many $i\in I$ such that $\diam U_i>\varepsilon$. This follows from the fact that the regions $U_i$, $i\in I$, are uniform quasidisks, so $(\diam U_i)^2$ is bounded by a uniform constant times the area of $U_i$; see the discussion on fat sets in Section \ref{section:cofat}.
\end{enumerate}
\begin{problem}\label{problem:schottky_topological}
If the Jordan regions $\{U_i\}_{i\in I}$ satisfy the above three conditions, does there exist a homeomorphism of $\widehat \C$ that maps the set $S=\widehat \C\setminus \bigcup_{i\in I}U_i$ onto a Schottky set?
\end{problem}
An affirmative answer is only known in the case that $\bar U_i\cap \bar U_j=\emptyset$ for $i\neq j$ and is due to Whyburn \cite{Whyburn:theorem}. The general case is expected to be much harder, since strong rigidity properties come into effect when we allow $\bar {U_i}\cap \bar {U_j}$ to be non-empty. In fact, in this case there might exist a \textit{unique} homeomorphism (up to M\"obius transformations) from $S$ onto a Schottky set, as illustrated by the last part of Theorem \ref{theorem:luo_ntalampekos}.

What remains to be done, in order to give a complete geometric characterization of sets that are quasiconformally equivalent to Schottky sets, is to characterize geometrically pairs of quasidisks that can be quasiconformally mapped to pairs of disks. Therefore, we are led back to Problem \ref{problem:annuli}, a solution of which is presented in the next section.

\subsection{Quasiconformal annuli revisited}\label{section:annuli_revisit}

If $U$ is a domain in $\widehat \C$ we denote by $U^*$ the complement of $\bar U$. If $\partial U$ has at least three points, then we denote by $h_U$ the hyperbolic metric\index{hyperbolic metric}\index{metric!hyperbolic} on $U$. 

Let $U,V\subset \widehat \C$ be Jordan regions such that $\bar U\cap \bar V$ contains at most one point. Following \cite{Ntalampekos:tangent} we define the \textit{relative hyperbolic metric}\index{metric!relative hyperbolic}\index{relative hyperbolic metric} of the pair $(V,U)$ as
$$d_{V,U}=\inf_\gamma \ell_{h_{U^*}}(\gamma),\quad z,w\in \partial V\setminus \bar U,$$ 
where the infimum is taken over all curves $\gamma$ in $U^*\setminus V$ that connect $z$ and $w$; see Figure \ref{fig:relative_hyperbolic}. We note that if $V$ is a quasidisk then $d_{V,U}$ is a metric on $\partial V\setminus \bar U$.

\begin{figure}
	\centering
	\includegraphics[scale=1]{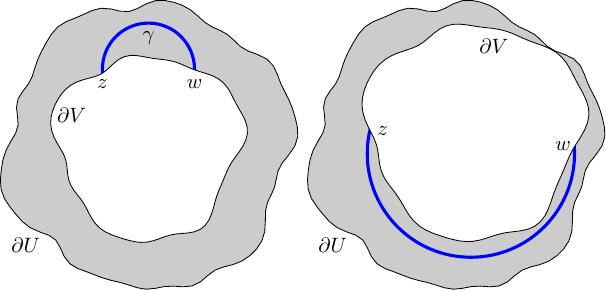}
	\caption{A curve $\gamma$ in $U^*\setminus V$ connecting $z,w\in \partial V\setminus \bar U$. Here $U$ is the unbounded region and $V$ is the bounded white region.}\label{fig:relative_hyperbolic}
\end{figure}

The metric $d_{V,U}$ can be explicitly computed when $U,V$ are disjoint disks. Specifically, if $V$ and $U^*$ are concentric disks in $\C$, then for all $z,w\in \partial V$ we have
$$C^{-1}\frac{|z-w|}{\dist(\partial U,\partial V)}\leq d_{V,U}(z,w) \leq C\frac{|z-w|}{\dist(\partial U,\partial V)},$$
where $C\geq 1$ is a uniform constant. Also, if $U,V$ are parallel half-planes (that is, disks tangent at $\infty$), then for all $z,w\in \partial V\setminus \bar U$ we have
$$d_{V,U}(z,w) =\frac{|z-w|}{\dist(\partial U,\partial V)}.$$
Thus, in these two model cases the metric $d_{V,U}$ is essentially a multiple of the Euclidean distance. 

The next two results from \cite{Ntalampekos:tangent} present the exact relation that is required between $d_{V,U}$ and the Euclidean distance so that a pair of disjoint quasidisks $U,V$ can be mapped to a pair of disks with a quasiconformal homeomorphism of $\widehat \C$. 

The first result concerns quasidisks $U,V$ that are ``tangent" to each other. By normalizing with M\"obius transformations, we assume that the point of ``tangency" is at $\infty$. Thus, we assume that $U,V\subset \C$ are unbounded quasidisks in the topology of $\C$. In the next statement we use the topology of $\C$. 

\begin{theorem}[\cite{Ntalampekos:tangent}]\label{theorem:ntalampekos_tangent}
Let $U,V\subset \C$ be unbounded quasidisks such that $\bar U\cap \bar V=\emptyset$ (in $\C$). There exists a quasiconformal map $f\colon \C\to \C$ that maps $U$ and $V$ to half-planes if and only if the identity map
$$\id\colon (\partial V,d_{V,U})\to (\partial V,|\cdot|)$$
is quasisymmetric. The statement is quantitative.
\end{theorem}
Recall the definition of a quasisymmetric map\index{quasisymmetric map}\index{map!quasisymmetric} from Section \ref{section:cospread}. The statement is quantitative in the sense that if the regions $U,V$ are $L$-quasidisks and the map $\id\colon (\partial V,d_{V,U})\to (\partial V,|\cdot|)$ is $\eta$-quasisymmetric, then there exists a $K(L,\eta)$-quasiconformal homeomorphism of $\C$ that maps $U$ and $V$ to half-planes. Conversely, if $U,V$ are $L$-quasidisks and there exists a $K$-quasiconformal homeomorphism as above, then the map $\id\colon (\partial V,d_{V,U})\to (\partial V,|\cdot|)$ is $\eta$-quasisymmetric for some distortion function $\eta$ that depends only on $K$ and $L$.

Next, we state the corresponding result for quasidisks with disjoint closures. Here we use the topology of the sphere $\widehat \C$. The statement is quantitative in the same sense as above.

\begin{theorem}[\cite{Ntalampekos:tangent}]\label{theorem:ntalampekos_disjoint}
Let $U,V\subset \widehat{\C}$ be quasidisks such that $\bar U\cap \bar V=\emptyset$. There exists a quasiconformal map $f\colon \widehat \C\to \widehat{\C}$ that maps $U$ and $V$ to disks if and only if the identity map
$$\id\colon (\partial V,d_{V,U})\to (\partial V,\chi)$$
is quasi-M\"obius. The statement is quantitative. 
\end{theorem}
Here $\chi$ denotes the chordal metric on $\widehat \C$, defined by 
\begin{align*}
\chi(z,w)= \frac{2|z-w|}{\sqrt{1+|z|^2}\sqrt{1+|w|^2}} \,\,\, \text{for $z,w\in\C$ and}\,\,\, \chi(z,\infty)=\frac{2}{\sqrt{1+|z|^2}}\,\,\, \text{for $z\in \C$}.
\end{align*}
Equivalently, one may use the spherical metric, which is bi-Lipschitz equivalent to the chordal one. Now we discuss the notion of a quasi-M\"obius map, as introduced by V\"ais\"al\"a \cite{Vaisala:quasimobius}. The \textit{cross ratio}\index{cross ratio} of a quadruple of distinct points $a,b,c,d$ in a metric space $(X,d)$ is defined as
$$[a,b,c,d]=\frac{d(a,c)d(b,d)}{d(a,d)d(b,c)}.$$
If $X=\widehat \C$, equipped with the chordal metric and $a,b,c,d\neq \infty$, then we have
$$[a,b,c,d]=\frac{|a-c||b-d|}{|a-d||b-c|}.$$
If one of the points is equal to $\infty$, then the factors containing that point are omitted. For instance,
$$[a,b,c,\infty]= \frac{|a-c|}{|b-c|}.$$

It is well-known that M\"obius transformations preserve cross ratios. Quasi-M\"obius maps are defined by requiring that they distort cross ratios in a controlled way. Specifically, a homeomorphism $f\colon X\to Y$ between metric spaces is \textit{quasi-M\"obius}\index{quasi-M\"obius map}\index{map!quasi-M\"obius} if there exists a homeomorphism $\eta\colon [0,\infty)\to [0,\infty)$ such that 
$$[f(a),f(b),f(c),f(d)]\leq \eta([a,b,c,d])$$
for every quadruple of distinct points $a,b,c,d\in X$. In this case we say that $f$ is $\eta$-quasi-M\"obius. Quasisymmetric maps are quasi-M\"obius and the converse is also true under some restrictions and normalizations. 

The reason for using quasi-M\"obius rather than quasisymmetric maps in Theorem \ref{theorem:ntalampekos_disjoint} is so that the statement becomes quantitative. We remark that \textit{quantitative dependence} is precisely what Problem \ref{problem:annuli} asks for. If we drop this requirement, then it would be elementary to map two quasidisks with disjoint closures onto two disks with a $K$-quasiconformal map, for a possibly very large $K$.

We close with an application of Theorem \ref{theorem:ntalampekos_tangent} in the case that $\partial V$ is the real line and $\partial U$ is the graph of a Lipschitz function.

\begin{theorem}[\cite{Ntalampekos:tangent}]\label{theorem:ntalampekos_lipschitz}
Let $f\colon \R\to (0,\infty)$ be a Lipschitz function. There exists a quasiconformal homeomorphism of $\C$ that preserves the real line and maps the graph of $f$ onto a line if and only if an antiderivative of $1/f$ is quasisymmetric. The statement is quantitative.
\end{theorem}

See Figure \ref{fig:lipschitz} for an illustration. The theorem applies the Lipschitz function 
$$f(x)=\frac{1}{(|x|+1)^p},\,\,\, x\in \R,$$
where $p>-1$. Every antiderivative of $1/f$ is a power function, which can be shown to be quasisymmetric (see \cite{Heinonen:metric}*{Exercise 10.3}). Thus, by Theorem \ref{theorem:ntalampekos_lipschitz}, the graph of $f$ can be straightened to a line by a quasiconformal map of $\C$ that preserves the real line. On the other hand, for the Lipschitz function $g(x)=e^{-|x|}$ the antiderivatives of $1/g$ are not quasisymmetric, so there is no quasiconformal map that preserves the real line and straightens the graph of $g$ to a line. 

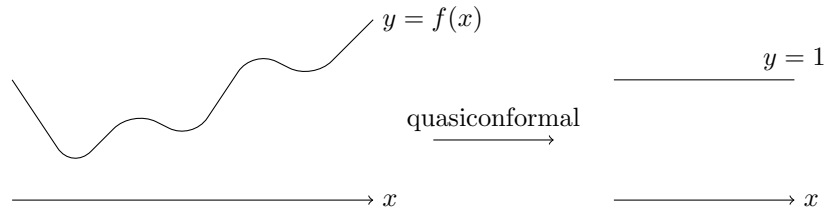
\begin{figure}
\begin{tikzpicture}[scale=0.8]
	\draw[->] (-3,0)--(3,0) node[right]{$x$};
	
	\draw[rounded corners=10pt] (-3,2)--(-2,0.5)--(-1,1.5)--(0,1)--(1,2.5)--(2,2)--(3,3) node[right] {$y=f(x)$};

	\draw[->] (4,1)--(6,1) node[pos=.5,above]{quasiconformal};
	\draw[->] (7,0)--(10,0) node[right]{$x$};
	\draw (7,2)--(10,2) node[above]{$y=1$};		
\end{tikzpicture}
	\caption{Illustration of Theorem \ref{theorem:ntalampekos_lipschitz}.}\label{fig:lipschitz}
\end{figure}

\bibliography{../../../biblio}
%\printindex
\end{document}